\documentclass[10pt]{article}

\usepackage{amsmath}
\usepackage{amssymb}
\usepackage{amsthm}
\usepackage[dvipsnames]{xcolor}
\usepackage{enumitem}
\usepackage{float}
\usepackage{hyperref}
\usepackage{mathtools}
\usepackage{mathrsfs}
\usepackage{microtype}

\usepackage{caption}
\usepackage{subcaption}
\usepackage{todonotes}
\usepackage{graphicx}
\graphicspath{ {./plots/} }

\usepackage[a4paper, total={6in, 9in}]{geometry}
\hypersetup{
colorlinks=true,
linkcolor=red,
filecolor=magenta,
urlcolor=cyan,
citecolor=green,
linktoc=all,
}

\newcommand{\E}{\mathbb{E}}
\newcommand{\F}{\mathbb{F}}
\newcommand{\N}{\mathbb{N}}
\renewcommand{\P}{\mathbb{P}}

\newcommand{\R}{\mathbb{R}}
\renewcommand{\S}{\mathbb{S}}

\newcommand{\cA}{\mathcal{A}}

\newcommand{\cC}{\mathcal{C}}
\newcommand{\cD}{\mathcal{D}}

\newcommand{\cJ}{\mathcal{J}}
\newcommand{\cL}{\mathcal{L}}

\newcommand{\cN}{\mathcal{N}}

\newcommand{\cS}{\mathcal{S}}
\newcommand{\cU}{\mathcal{U}}

\newcommand{\cY}{\mathcal{Y}}

\newcommand{\fh}{\hat{h}}

\newcommand{\rD}{\mathrm{D}}
\newcommand{\rd}{\mathrm{d}}

\newcommand{\argdot}{\,\cdot\,}
\newcommand{\equi}{\gamma_{\mathrm{eq}}}
\newcommand{\hEqui}{h_{\mathrm{eq}}}
\newcommand{\vEqui}{v_{\mathrm{eq}}}
\newcommand{\pEqui}{p_{\mathrm{eq}}}

\newcommand{\eps}{\varepsilon}

\theoremstyle{plain}
\newtheorem{theorem}{Theorem}[section]
\newtheorem{proposition}[theorem]{Proposition}
\newtheorem{corollary}[theorem]{Corollary}
\newtheorem{lemma}[theorem]{Lemma}

\newtheorem{assumption}{Assumption}

\theoremstyle{definition}
\newtheorem{remark}[theorem]{Remark}

\newcommand{\DM}{{decision maker}}

\title{Continuous-Time Dynamic Decision Making with Costly Information}

\author{
Christoph Knochenhauer\thanks{
Technical University of Munich, School for Computation, Information and Technology, Department of Mathematics,
Parkring 11, 85748 Garching bei M\"{u}nchen, Germany
 ({\tt knochenhauer@tum.de})}
\and
Alexander Merkel\thanks{
Technical University of Berlin, Institute of Mathematics,
Str.\ des 17.\ Juni 136, 10587 Berlin, Germany
 ({\tt merkel@math.tu-berlin.de})}
\and
Yufei Zhang
\thanks{Department of Mathematics,
Imperial College London, 
  United Kingdom
({\tt yufei.zhang@imperial.ac.uk})}
}

\begin{document}
\maketitle

\begin{abstract}
We consider a continuous-time linear-quadratic Gaussian control problem with partial
observations and costly information acquisition.
More precisely, we assume the drift of the state process to be governed by an unobservable
Ornstein--Uhlenbeck process.
The decision maker can additionally acquire information on the hidden state
by conducting costly tests, thereby augmenting the available information.

Combining the Kalman--Bucy filter with a dynamic programming approach,
we show that the problem can be reduced to
a deterministic control problem for the conditional variance of the unobservable state.
Optimal controls and value functions are
derived in a semi-explicit form, and we present an extensive study of the qualitative
properties of the model.

We demonstrate that both the optimal cost and the marginal cost increase with model uncertainty.
We identify a critical threshold: below this level, it is optimal not to acquire additional information,
whereas above this threshold, continuous information acquisition is optimal,
with the rate increasing as uncertainty grows.
For quadratic information costs, we derive the precise asymptotic behavior
of the acquisition rate as uncertainty approaches zero and infinity.

\end{abstract}

\section{Introduction}
\label{sec:intro}
This paper investigates a continuous-time stochastic optimal control problem with partial observations
in which the decision maker can continuously choose to acquire costly information
on the unobservable state.
The motivation for studying this model is well described in \cite{xu2023decision},
where the authors describe the current state of research in this area as follows:
\begin{quote}
	``To our knowledge, the literature on models that combine sequential decision-making
	with dynamically changing costly information sources is not adequately developed.
	Indeed, the existing models predominantly consider both a single information source
	and a single decision occurring at an (optimal) time,
	at which the problem immediately terminates.''
\end{quote}
With this in mind, the motivation of this article is to develop a tractable
continuous-time model of dynamic decision making with costly controlled information.
Such problems are, in general, mathematically quite challenging due to the explicit
dependence of the observation filtration on the control; we refer to \cite{cohen2023optimal}
for recent work in this direction.
One of the main contributions of this article
is that we are able to derive a non-trivial model which is tractable enough to
allow for semi-explicit solutions and for an extensive analysis of the qualitative
properties of the model.

To make this more precise, in this paper we consider
a continuous-time state process of the form
\[
  \rd X^u_t = \bigl(\mu_t + u_t) \rd t + \sigma_1 \rd B^1_t,
\]
where $u$ is the control and $\mu$ is an unobservable Ornstein--Uhlenbeck process with dynamics
\[
  \rd \mu_t = \lambda\bigl(\bar\mu - \mu_t\bigr)\rd t + \sigma_2\rd B^2_t.
\]
A central feature of our model is that the decision maker has access to additional costly
information on $\mu$ via an auxiliary controlled state variable with dynamics
\begin{equation}\label{eq:intro_aux_state}
  \rd Y_t^h = \mu_t\sqrt{h_t}\rd t + \rd B^3_t,
\end{equation}
where $h$ is the information acquisition rate.
Intuitively, choosing a large value for the control $h$ increases the ratio of the
signal $\mu\sqrt{h}\rd t$ to the noise $\rd B^3$, hence revealing additional information
about $\mu$. 
The dynamics \eqref{eq:intro_aux_state} can be interpreted as conducting independent tests
on the hidden state $\mu$, similar to the approach in \cite{moscarini2001optimal, ekstrom2024detection},
with $h$ representing the frequency of these tests within a short time interval;
see Section \ref{sec:theOP} for more details.
The decision maker's objective is to minimize the expected cost of the form
\begin{equation}
\label{eq:intro_cost}
  \E\Bigl[\int_0^\infty e^{-\delta t}\Bigl(
    \frac{1}{2}\bigl(\kappa|X_t^u|^2 + \rho|u_t|^2\bigr) + c(h_t)\Bigr)\rd t\Bigr]
\end{equation}
for $\kappa, \rho>0$ and a convex cost function $c$ for the information acquisition rate $h$.
In other words, the goal is to steer
$X^u$ close to zero while balancing the trade-off between the quadratic cost for the state process 
$X^u$, and state control $u$,
and the additional cost associated with for information acquisition via $h$. 

Here, we emphasize that the information control $h$ multiplicatively affects the drift
of the observation dynamics \eqref{eq:intro_aux_state},
distinguishing this information acquisition control problem from the classical partially observed
linear-quadratic (LQG) control problem, where the observation process is typically either
uncontrolled (see e.g., \cite{Bensoussan_1992, wang2015linear}) or controlled linearly
\cite{barzykin2024unwinding}. 
The nonlinearity of $h$ in \eqref{eq:intro_aux_state} implies that the optimal information
acquisition rate $h$ is no longer a simple linear function of the conditional mean
of $\mu$, as seen in classical settings.
Instead, the optimal acquisition rate must account for the entire conditional distribution of
$\mu$ given the observations, making its characterization more complex.
 
While our problem formulation is kept abstract, LQG control
problems have many applications in essentially all areas in which controlled
dynamical systems under uncertainty play a role.
This includes but is not limited to robotics, economics, finance, engineering, physics, medicine,
and many other fields.
For example, let us specifically mention the connection to optimal execution
problems in mathematical finance \cite{cartea2015algorithmic, gueant2016financial}
and highlight the relevance of dynamic costly information acquisition in that context.
The objective in optimal execution problems is to liquidate a large
financial position in the presence of price impact, that is, in a situation in which
large trades lead to price movements which are adverse to the trader's objective.
Many models of optimal execution are formulated as linear-quadratic stochastic control
problems.
We furthermore specifically mention the articles
\cite{belak2018liquidation, lehalle2019incorporating} in which the authors
consider price processes involving signals, that is, price processes having a drift.
While these articles assume the signals to be observable, in practice the signals
are unobservable, and significant effort has to be made to obtain good estimates.
In particular, acquiring information in this setting is costly
(hiring staff, acquiring additional data, etc.), and the decision
maker can choose how much wealth to spend on improving the quality of information.
As such, costly information acquisition should be explicitly included in the
problem formulation, emphasizing the need to get a better theoretical understanding
of the effects of costly information acquisition in control problems with partial
observations.

Our model is furthermore related to the fields of reinforcement learning and adaptive
control.
In these fields, the control of the state process and the control of
information are often identical \cite{szpruch2021exploration, basei2022logarithmic},
while in our setting of costly information acquisition
the controls for the state process and the information acquisition are decoupled.
An extensive literature review of related work can be found in \cite{xu2023decision}.

To the best of our knowledge, our model is the first instance of a continuous-time
stochastic control problem with dynamic, costly information acquisition.
In \cite{xu2023decision}, information is acquired passively and is embedded in optimal stopping problems.
Similarly, in \cite{moscarini2001optimal,dalang2015quickest,
zhong2022optimal,ekstrom2024detection} the information acquisition is dynamic, and
the control problem under consideration is a problem of optimal stopping.

The mechanism of costly and controlled information acquisition that we consider in this
article goes back to \cite{moscarini2001optimal} and has been further considered,
for example, in \cite{dalang2015quickest,ekstrom2024detection}.
There, however, the authors consider problems involving an unobservable two-state
random variable, whereas we consider a continuously evolving hidden Ornstein--Uhlenbeck process.
In the case of a static hidden random variable, the uncertainty is
decreasing over time as there is no additional exogenous uncertainty injected into the system over time.
This is in contrast to our model where the hidden state is a diffusion process,
and hence the amount of uncertainty on the hidden state depends on the rate of information acquisition.
Next, in \cite{dalang2015quickest,ekstrom2024detection},
the cost is linear in the rate of information acquisition
and therefore the optimal control is of bang-bang type,
whereas we consider a strictly convex cost structure $c(h)$ as, for example, in
\cite{moscarini2001optimal,zhong2022optimal}.
This is well-motivated by strictly increasing marginal cost for information acquisition.

To solve the problem considered in this article, we proceed as follows.
We begin in Section~\ref{sec:theOP} with a rigorous formulation of the optimization
problem, including a discussion motivating the form of the observation 
process~\eqref{eq:intro_aux_state} and a careful description of the decision makers
information structure.

In Section~\ref{sec:transformation}, we apply classical techniques from filtering theory
to embed the original partial observation problem into an auxiliary full information
problem with a larger state space.
The main tool for the transformation is the Kalman--Bucy filter.
We emphasize that while stochastic control problems with control-dependent information are often formulated in a weak formulation to improve tractability (see, e.g., \cite{benevs1991separation}),
our approach successfully addresses the problem in the technically more
challenging strong formulation (see Remark \ref{rmk:weak_strong} for details).

The auxiliary full information problem consists of three states given by the original
state process $X^u$, the conditional mean and the conditional variance
of the hidden Ornstein--Uhlenbeck process $\mu$.
In Section~\ref{sec:reduction}
we show that the problem can be reduced to a nonlinear deterministic control problem
for the conditional variance alone.
This reduction is achieved by constructing
a candidate solution of the Hamilton--Jacobi--Bellman (HJB) equation of the full information
problem which is explicit in all state variables except for the conditional variance.
This leads to a reduction of the HJB equation, and the reduced equation itself is
an HJB equation corresponding to a control problem for the conditional variance.

In Section~\ref{sec:reduced-control-problem} we establish the continuous differentiability of the value
function $v$ of the reduced control problem for the conditional variance
(Corollary \ref{cor:regularity-original-value-function}), which subsequently
allows for the construction of optimal controls for both
the reduced problem and the full information problem using verification arguments. 
The regularity of the reduced value function $v$ is achieved through a control-theoretic approach,
rather than relying on standard PDE methods.  
This is because the associated HJB equation is first-order and lacks elliptic regularity.
In particular, we first show that the reduced problem
admits a continuous optimal control using a compactification argument.
Using this together with the observation that
the controlled state process is twice differentiable in the initial state,
we show that the value function is increasing and strictly concave.
We then reparametrize the problem
in terms of the conditional precision (the inverse of the conditional variance),
and show that the reparametrized value function is strictly convex.
Since concavity in the conditional variance implies semi-concavity in the conditional precision,
it follows that the value function in both parametrizations is continuously differentiable.
This, along with a verification argument, enables the construction of a unique continuous
optimal control in feedback form. 

We then extensively analyze our model's behavior under the optimal information acquisition control
in Sections~\ref{sec:reduced-control-problem} and \ref{sec:properties}.
Specifically, we demonstrate that both the control cost and the marginal control cost increase
as the model's uncertainty increases (Remark \ref{rmk:concavity}).
Moreover, we identify a critical threshold where, if the uncertainty in the hidden state $\mu$
(measured by the conditional variance) is below this level,
it is optimal not to acquire additional information.
Conversely, when uncertainty surpasses this threshold,
continuous information acquisition becomes optimal, with the acquisition rate increasing in uncertainty
(Lemma \ref{lem:feedback-map}).
We show that this critical threshold is zero if the cost function 
$c$ in \eqref{eq:intro_cost} is of a power type, indicating that it is optimal to acquire information
continuously (Remark \ref{rmk:power_threshhold}). 
For a quadratic cost $c$, the information acquisition rate vanishes quadratically
as the conditional variance approaches zero and is asymptotically proportional to the conditional 
standard deviation as the conditional variance approaches infinity
(Theorem \ref{eq:quadratic_asymptotic}).

Finally, we prove that the optimal conditional variance and the corresponding optimal information acquisition
rate converge monotonically to an equilibrium level, regardless of the initial uncertainty of the system (Proposition \ref{prop:properties-equilibrium-state-control}).
We also explicitly characterize the optimal value function and feedback control near the equilibrium (Proposition \ref{prop:value-equilibrium}). 
The monotonicity of the equilibrium with respect to the model parameters is obtained through a sensitivity analysis (Section \ref{subsec:sensitivity}).
Numerical experiments are presented to compare our model with two benchmark problems:
one where the decision maker fully observes the hidden state $\mu$
and another where the state $\mu$ is unobserved but no additional information can be acquired
(Section \ref{sec:properties-original}).
They further serve as illustrations for the qualitative properties.

\section{Problem Formulation}\label{sec:theOP}
Throughout the article, all processes are defined on the time interval $[0,\infty)$
unless explicitly stated otherwise.
We fix a complete probability space $(\Omega,\mathcal{F},\P)$
supporting a three-dimensional standard Brownian motion $B=(B^{1}, B^{2}, B^{3})$ and an independent random variable $\mu_{0}$ with Gaussian distribution
$\cN(m_{0},\gamma_{0})$ under $\P$, where $m_{0}\in\R,\gamma_{0} \geq 0$.
We denote by $\F=(\mathcal{F}_{t})_{t\in[0,\infty)}$ the filtration generated by $\mu_{0}$
and $B$ augmented by the $\P$-nullsets, so that $\F$ satisfies the usual assumptions
of right-continuity and completeness.

We consider a decision maker who chooses a control
$u:\Omega\times[0,\infty)\to\R$ to control the state
$(X^{u},\mu)$ governed by the stochastic differential equation
\begin{align}\label{eq:po-state}
	\rd \begin{pmatrix}
		X^{u}_{t} \\ \mu_{t}
	\end{pmatrix}
	=\begin{pmatrix}
		 \mu_{t}+u_{t} \\ \lambda(\bar{\mu}-\mu_{t})
	 \end{pmatrix}\rd t
	+ \begin{pmatrix}
		  \sigma_{1} & 0 \\ 0 & \sigma_{2}
	  \end{pmatrix} \rd \begin{pmatrix}
		B_{t}^{1} \\ B_{t}^{2}
	\end{pmatrix},\quad
	\begin{pmatrix}
		X_{0}^{u} \\ \mu_{0}
	\end{pmatrix}
	=\begin{pmatrix}
		 x_{0} \\ \mu_{0}
	\end{pmatrix},
	\quad t\in[0,\infty),
\end{align}
where $\lambda \geq 0$ is the speed of mean reversion,
$\bar{\mu}\in\R$ is the long-run mean,
$\sigma_{1} > 0$ and $\sigma_{2}\geq0$ are volatility coefficients,
and $x_{0}\in\R$ is the initial state of $X^{u}$.
In the following, the \DM\ does not observe the Ornstein-Uhlenbeck process $\mu$ directly.
For that reason, we assume a strictly positive diffusion coefficient $\sigma_1>0$ 
so that the \DM\ cannot recover the hidden drift $\mu_t$ by observing the state $X^u_t$
and the control $u_t$.

We allow the \DM\ to acquire additional information on the hidden drift $\mu$ by
conducting independent {\it tests} as in \cite{moscarini2001optimal,ekstrom2024detection}
revealing additional, albeit noisy information.
To motivate our model, let us temporarily take as given an entire sequence $(W^i)_{i\in\N}$ of independent Brownian motions additionally independent of $\mu_0$.
The \DM\ can choose to conduct $n\in\N$ independent tests for $\mu$,
resulting in observations $Y^i$, $i=1,\dots,n$, with
\[
	\rd Y^{i}_{t} = \mu_{t} \rd t + \rd W^{i}_{t},\quad Y^{i}_{0}=0,
	\quad t\in[0,\infty).
\]
With this, the weighted sum of tests $\bar{Y}^n := (Y^{1},\dots,Y^{n})/\sqrt{n}$ for the drift $\mu$ is seen to satisfy
\begin{equation*}
	\rd \bar{Y}^n_{t} = \mu_{t}\sqrt{n}  \rd t + \rd \bar{W}^n_{t},\quad \bar{Y}^n_{0}=0,
	\quad t\in[0,\infty),
\end{equation*}
with $\bar{W}^n:=(W^{1} + \dots + W^{n})/\sqrt{n}$ being a one-dimensional standard Brownian motion.

In our model, we allow the \DM\ to choose the number of tests dynamically in time and in arbitrary 
quantities by replacing the fixed number $n$ by a {\it rate of information acquisition} $h:\Omega\times[0,\infty)\to [0,\infty)$, leading to the observation process
\begin{equation*}
	\rd Y_{t}^{h} = \mu_{t} \sqrt{h_{t}} \rd t + \rd B^{3}_{t},\quad Y_{0}^h=0,
	\quad t\in[0,\infty),
\end{equation*}
where $B^3$ is the Brownian motion we fixed in the beginning of this section.
We see that, heuristically, increasing the number of tests $h$ increases the signal-to-noise ratio
in $Y^h$, where $\mu\sqrt{h}$ is the signal and the Brownian motion $B^3$ is the noise.

The \DM\ is allowed to choose a {\it state control} $u$ and an {\it information acquisition rate} $h$.
More precisely, we let  $\cU:=\R\times[0,\infty)$ and define the set of
{\it pre-admissible controls} $\cA^{pre}$ as those $\F$-progressively measurable pairs
of $\cU$-valued processes $(u,h)$ such that for all $t>0$,
\begin{equation}\label{eq:pre-admissible}
  \E\left[\int_0^t \bigl|u(s)\bigr|\rd s\right] < \infty
  \qquad\text{and}\qquad
  \lim_{t\to\infty}e^{-\delta t} \E\left[(X^{u}_{t})^{2} + \int_0^t h_s \rd s\right] = 0.
\end{equation}
The state process associated with a pre-admissible control $(u,h)$ is the three-dimensional process $(X^u,\mu,Y^h)$ with dynamics
\begin{equation}\begin{split}\label{eq:state-initial}
	\rd \begin{pmatrix}
		X_{t}^{u} \\ \mu_{t} \\ Y_{t}^{h}
	\end{pmatrix}
	=\begin{pmatrix}
		\mu_{t} + u_{t} \\ \lambda(\bar{\mu}-\mu_{t}) \\ \mu_{t} \sqrt{h_{t}}
	\end{pmatrix}\rd t
	+ \begin{pmatrix}
		\sigma_{1} & 0 & 0 \\ 0 & \sigma_{2} & 0 \\ 0 & 0 & 1
	\end{pmatrix}\rd B_{t},\quad
	\begin{pmatrix}
		X_{0}^{u} \\ \mu_{0} \\ Y_{0}^{h}
	\end{pmatrix}
	=\begin{pmatrix}
		x \\ \mu_{0} \\ 0
	\end{pmatrix},
	\quad t\in[0,\infty).
\end{split}\end{equation}
The actions of the \DM are required to be based on only the information available by observing
the controlled state $X^{u}$ and the observation process $Y^{h}$.
This is formalized by restricting the set of admissible controls $\cA^{pre}$
to those which are adapted to the {\it observation filtration} $\cY^{(u,h)}$ generated by $(X^{u},Y^{h})$ 
and augmented by the $\P$-nullsets.
By definition of $\cA^{pre}$, it is clear that $\cY^{(u,h)}\subseteq\F$
for any pre-admissible control $(u,h)\in\cA^{pre}$.
With this, we define
\[
\cA := \Bigl\{(u,h)\in\cA^{pre} \;\Big|\; (u,h)\text{ is }\cY^{u,h}\text{-progressively measurable}\Bigr\}
\]
as the set of admissible controls.

\begin{remark}
  Note that some care has to be taken when defining the set of admissible controls
  as the SDE \eqref{eq:state-initial} for the state process only makes sense
  for progressively measurable controls $(u,h)$, but the filtration $\cY^{(u,h)}$ is only defined
  after the SDE has been solved.
  Similarly to \cite{cohen2023optimal}, we prevent this circular dependence between the existence
  of states and observation filtration by defining the state process first for $\F$-progressively
  measurable controls in $\cA^{pre}$ and then restricting to the $\cY^{(u,h)}$-progressively
  measurable controls in $\cA$.
  We furthermore point out that the set of controls is not very tractable as,
  for example, for $(u_1,h_1),(u_2,h_2)\in\cA$ we need not have $(u_{1}+u_{2},h_{1}+h_{2})\in\cA$.
\end{remark}

\begin{remark}\label{remark:filtration}
Observe that for any $(u,h)\in\cA^{pre}$ we can write
\[
 X^u_t - \int_0^t u_s \rd s = x + \int_0^t \mu_s\rd s + \sigma_1 B^1_t = X^0_t,\quad t\in[0,\infty),
\]
from which we conclude that $\cY^{0,h}\subseteq\cY^{u,h}$,
so the state control $u$ potentially reveals additional information on $\mu$.
However, we shall see in Lemma~\ref{lem:optimal-state-full} that for the optimal control $(u^*,h^*)$
we also have $\cY^{u^*,h^*}\subseteq\cY^{0,h^*}$, meaning that there indeed is a separation in the role
of the two controls $u^*$ and $h^*$ in the sense that the sole purpose of $u^*$ is controlling the state whereas the sole purpose of $h^*$ is the acquisition of information.
In this sense, the {\it separation principle} of stochastic control theory with partial information holds;
see \cite{wonham1968separation,bensoussan1992stochastic,georgiou2013separation}.
\end{remark}

The performance of an admissible control $(u,h)\in \cA$ is given by the infinite horizon cost functional 
\begin{align*}
J(u,h):=\E\left[\int_{0}^{\infty}e^{-\delta t}\left(
	\frac{1}{2}\big(\kappa |X^{u}_{t}|^{2} + \rho |u_{t}|^{2} \big)+ c(h_{t})\right) \rd t \right]
	\quad 	 \textnormal{subject to \eqref{eq:state-initial}},
\end{align*}
where $\kappa,\rho>0$ are the cost coefficients for state and control, respectively,
$\delta>0$ is the discount factor, and $c: [0,\infty)\rightarrow[0,\infty)$ is
the cost of information acquisition.
This functional $J$ reflects the trade-off between a quadratic cost for deviations of the state process 
$X^u$, the use of the control $u$,
and an additional cost $c$ on the information acquisition rate $h$.
In all that follows, we impose the following standing assumption on the information acquisition cost. 
\begin{assumption}\label{ass:cost}
	The function $c :[0,\infty)\to[0,\infty)$ is strictly convex,
    twice continuously differentiable, and strictly increasing.
\end{assumption}

The convexity of the cost function implies that the marginal cost of generating
additional independent tests is increasing.
This assumption is motivated, for example, by \cite{moscarini2001optimal, blatter2012costs}.
A typical example of a cost function satisfying the standing assumption is the power cost function
\begin{equation}\label{eq:cost-functions}
  c(x) = \zeta x^{1+\epsilon},\quad x\in[0,\infty),
\end{equation}
for constants $\zeta,\epsilon>0$.
The choice of $\epsilon=1$ corresponds to the special case of quadratic cost of information acquisition.

The \DM\ aims to minimize the cost functional $J(u,h)$ over all of admissible controls $(u,h)\in\cA$.
We subsequently consider the {\it partial information problem}
\begin{equation}\label{eq:original-control-problem}
	\inf_{(u,h)\in\cA} J(u,h).
\end{equation}
As highlighted in Section \ref{sec:intro}, even in the special case of quadratic information costs,
the control problem \eqref{eq:original-control-problem} does not fit within the framework
of the classical partially observed LQG control problem studied in \cite{Bensoussan_1992, wang2015linear},
since the information control $h$ affects the drift
of the state dynamics \eqref{eq:state-initial} nonlinearly.

\begin{remark}
\label{rmk:weak_strong}
    We formulate the optimization problem in the strong formulation instead of the weak formulation
    and, as a result, the observation filtration $\cY^{u,h}$ is control-dependent.
    While the two formulations are often equivalent in the sense that they lead to the same optimal cost, 
    the same cannot always be said about the existence of optimal controls,
    especially in problems with controlled observations;
    see \cite{benevs1991separation} for an intriguing example.
    In our case, however, we are able to derive a semi-explicit
    feedback map for the optimal control, regular enough so that the associated optimally controlled state process even exists in the strong sense.
    Hence, the strong and weak formulation of our problem are indeed equivalent.
\end{remark}

\section{Transformation to Full Information}\label{sec:transformation}
In this section, we embed the original partial information problem into a full information problem amenable to dynamic programming.
Note that, in the original problem, the drift coefficients of the state process
and the Brownian motion $B$ are not necessarily adapted to the observation filtration
$\cY^{u,h}$ as, generally, $\mu$ is not $\cY^{u,h}$-adapted. 

This is resolved by applying the Kalman-Bucy filter \cite[Theorem 12.7]{liptser2013statisticsII}
and augmenting the state process by the conditional mean and variance of $\mu$.
This allows us to recast the original problem as an equivalent problem under full information.
The latter problem is then formulated dynamically and solved via dynamic programming techniques. 

Following \cite{liptser2013statisticsI}, the distribution of $\mu_t$ conditional on $\cY^{u,h}_t$
is Gaussian and hence characterized by its conditional mean and variance
\[
  m_t^{u,h}:=\E\bigl[\mu_{t}\big|\cY_{t}^{u,h}\bigr]
  \quad\text{and}\quad
  \gamma_t^{u,h}:=\E\bigl[(\mu_{t}-m_{t}^{u,h})^{2}\big|\cY_{t}^{u,h}\bigr].
\]
Moreover, (suitable versions of) these conditional expectations follow the dynamics
\begin{equation*}\begin{split}
		\rd \begin{pmatrix}
			m_{t}^{u,h} \\ \gamma_{t}^{u,h}
		\end{pmatrix}
		=\begin{pmatrix}
			\lambda(\bar{\mu} - m_{t}^{u,h}) \\
			f(\gamma_{t}^{u,h},h_{t})
		 \end{pmatrix}\rd t
		+ \begin{pmatrix}
			\bar{\sigma}_{1}\gamma_{t}^{u,h} & \sqrt{h_{t}}\gamma_{t}^{u,h} \\
			0 & 0
		  \end{pmatrix}\rd I^{u,h}_{t},\quad
		\begin{pmatrix}
			m_{0}^{u,h} \\ \gamma_{0}^{u,h}
		\end{pmatrix}
		= \begin{pmatrix}
			  m_{0} \\ \gamma_{0}
		  \end{pmatrix},
		  \quad t\in[0,\infty),
\end{split}\end{equation*}
where
\[
  \bar{\sigma}_{1}:= 1 / \sigma_{1},
\]
the function $f:\R^2\to\R$ is defined as
\begin{equation}\label{eq:def-f}
	f(\gamma,h)
	:=-(\bar{\sigma}_{1}^{2}+h) \gamma^{2}
	-2\lambda\gamma
	+\sigma_{2}^{2},\quad (\gamma,h)\in\R^2,
\end{equation}
and the process $I^{u,h} = (I^{1,u,h},I^{2,u,h})$ is the two-dimensional {\it innovations process} defined as
\begin{equation}\begin{split}\label{eq:def-innovations-process}
	\begin{pmatrix}
		I^{1,u,h}_{t} \\ I^{2,u,h}_{t}
	\end{pmatrix}
	:=\begin{pmatrix}
		\bar{\sigma}_{1}(X^{u}_{t} - x - \int_{0}^{t} (m_{s}^{u,h} - u_{s})\rd s) \\
		Y^{u,h}_{t} - \int_{0}^{t} m_{s}^{u,h}\sqrt{h_{s}}\rd s
	\end{pmatrix}
	=\begin{pmatrix}
		  B^{1}_{t} - \bar{\sigma}_{1}\int_{0}^{t}(m_{s}^{u,h} - \mu_{s})\rd s \\
		  B^{3}_{t} - \int_{0}^{t}(m_{s}^{u,h} - \mu_{s})\sqrt{h_{s}}\rd s
	  \end{pmatrix},\quad t\in[0,\infty).
\end{split}\end{equation}
Note that $f$ is quadratic in $\gamma$, and hence the conditional variance
$\gamma^{u,h}$ solves a Riccati differential equation controlled by the information acquisition rate $h$. 
Moreover, the dynamics of $\gamma^{u,h}$ do not depend on the state control $u$,
so that we may subsequently write $\gamma^h$ in place of $\gamma^{u,h}$ for ease of notation.

We now recall a classical result regarding this transformation, which allows the \DM\
to replace the unobservable state $\mu$ by its conditional mean and variance $(m^{u,h},\gamma^h)$
in the state dynamics and treat any remaining uncertainty as Brownian noise
with respect to the observation filtration $\cY^{u,h}$.
The classical proof can, for example, be found in
\cite[Lemma 11.3]{liptser2013statisticsII} or \cite[Lemma 22.1.7]{cohen2015stochastic}.
\begin{lemma}\label{lem:innovations-process}
	For any admissible control $(u,h)\in\cA$,
	the innovations process $I^{u,h}$ is a two-dimensional standard
	$(\cY^{u,h},\P)$-Brownian motion.
\end{lemma}
 
In light of the previous result and since the observation process $Y^{u,h}$ has the sole purpose
of generating information, but neither affects the dynamics of $X^u$ nor the cost functional $J$ directly,
the state of the system is fully described by the process $Z^{u,h}:=(X^{u},m^{u,h},\gamma^{h})$.
The main concern in the study of existence of $Z^{u,h}$ is whether the conditional variance
$\gamma^h$ blows up.
The following lemma shows that the conditional variance is uniformly bounded over all admissible controls. 

\begin{lemma}\label{lem:gamma-bounded}
For any initial value $\gamma_0\in[0,\infty)$, the uncontrolled conditional variance $\gamma^{0}$ exists as a continuously differentiable solution of
\[
    \rd\gamma^0_t = f(\gamma^0_t,0)\rd t
	,\quad \gamma^0_0 = \gamma_0,\quad t\in[0,\infty).
\]
Moreover, $\gamma^0$ is bounded and converges for $t\to\infty$ to the unique stable equilibrium
\begin{equation}\label{eq:gamma-infty}
	\gamma^{0}_{\infty}
	:= -\lambda\sigma_{1}^{2} + \sqrt{\lambda^{2}\sigma_{1}^{4} + \sigma_{1}^{2}\sigma_{2}^{2}} \geq 0.
\end{equation}
Finally, for all $(u,h)\in\cA^{pre}$, $\gamma^h$ exists as an absolutely continuous function satisfying
\[
  \rd\gamma^h_t = f(\gamma^h_t,h_t)\rd t,
  \quad \gamma^h_0 = \gamma_0,\quad t\in[0,\infty),
\]
almost everywhere and it holds that
\begin{equation}\label{eq:gamma-max}
  0\leq \gamma_{t}^{h}\leq \gamma_t^0 \leq \max\{\gamma^{0}_{\infty}, \gamma_{0}\},\quad t\in[0,\infty).
\end{equation}
\end{lemma}

\begin{proof}
First, observe that
\[
  f(\gamma,0) = 0\text{ for }\gamma\in[0,\infty)
  \qquad \text{if and only if} \qquad
  \gamma = -\lambda\sigma_{1}^{2} + \sqrt{\lambda^{2}\sigma_{1}^{4} + \sigma_{1}^{2}\sigma_{2}^{2}} = \gamma^{0}_{\infty},
\]
that is, $\gamma^{0}_{\infty}$ is the unique positive root of the quadratic function
$\gamma\mapsto f(\gamma,0)$.
From this, it also follows that $f(\gamma,0)\leq 0$ for all $\gamma\geq \gamma_\infty^0$.
Together with $f(0,0) = \sigma_2^2 \geq 0$ and the local Lipschitz continuity of $f$
we obtain the existence of a unique non-negative solution $\gamma^0$ bounded from above by
$\max\{\gamma^{0}_{\infty}, \gamma_{0}\}$.
Finally, the existence of $\gamma^h$ and the inequality $\gamma^h\leq\gamma^0$ follow
from the observation that $h\mapsto f(\gamma,h)$ is non-increasing, thus upper bounded by
$f(\gamma,0)$, and conversely, $f(0,h)\geq0$ for all $h\geq 0$.
Finally, the convergence of $\gamma^0_t$ to the equilibrium $\gamma^0_\infty$ as $t\to\infty$
is a standard result; see \cite[Remark 2.17, Theorem 2.44]{grass2008optimal}. 
\end{proof}

\begin{remark}\label{rem:negative-h}
If we replace the lower bound $h \geq 0$ with a (possibly negative) lower bound of the form
$h \geq -\bar{\sigma}_1^2 + \nu$ for some $\nu>0$, then $\gamma^h$ still exists and is bounded.
Indeed, in that case we have
\[
  f(0,h) = \sigma_2^2 \geq 0
  \quad\text{and}\quad
  f(\gamma,h) \leq f(\gamma,-\bar{\sigma}_1^2+\nu)
\]
for all $\gamma\in[0,\infty)$ and $h\geq-\bar{\sigma}_1^2+\nu$.
From this, it follows that $0\leq \gamma^h\leq \gamma^\nu$, where $\gamma^\nu$ is the bounded solution
of the differential equation with $h\equiv 0$ and $\bar{\sigma}_1^2$ replaced by $\nu$.
\end{remark}

In what follows, we fix a constant $\gamma_{\max} > \gamma^0_\infty$.
From the previous lemma, it follows that for any initial condition $\gamma_0\in[0,\gamma_{\max}]$
and any admissible information acquisition rate $h$,
the corresponding conditional variance process $\gamma^h$ never exits the interval $[0,\gamma_{\max}]$.
In particular, we may subsequently take the state space of $\gamma^h$ to be the compact interval $[0,\gamma_{\max}]$.

With this, we are now finally in a position to formulate the transformed
full information problem dynamically.
We begin by introducing the augmented state space $\S:=\R^{2}\times[0,\gamma_{\max}]$.
Next, for any control $(u,h)\in\cA$ and initial condition $z=(x,m,\gamma)\in\S$, the state process $Z^{u,h} = Z^{u,h;z}$ with $Z^{u,h}=(X^u,m^{u,h},\gamma^h)$ is defined as the unique $\S$-valued strong solution of
\begin{equation}\begin{split}\label{eq:state-aux}
		\rd Z_{t}^{u,h}
		=\begin{pmatrix}
			m^{u,h}_{t} + u_{t}\\
			\lambda\bigl(\bar{\mu} - m^{u,h}_{t}\bigr) \\
			f(\gamma_{t}^{h},h_{t})
		\end{pmatrix}\rd t
		+\begin{pmatrix}
			\sigma_{1} & 0 \\
			\bar{\sigma}_{1}\gamma_{t}^{h} & \sqrt{h_{t}}\gamma_{t}^{h}
			\\ 0 & 0
		\end{pmatrix}\rd I^{u,h}_{t},\quad
		Z_{0}^{u,h}=z,
		\quad t\in[0,\infty).
\end{split}\end{equation}
Here, existence of a solution is guaranteed by Lemma~\ref{lem:gamma-bounded} and the linearity of the drift coefficient of $Z^{u,h}$ in the first two components.
We furthermore note that, despite of what the dynamics specified in \eqref{eq:state-aux} suggest, the first component $X^u$ of $Z^{u,h}$ does not depend on the information acquisition rate $h$ since
\begin{equation*}
  \rd X^u_t = (m^{u,h}_{t} + u_{t})\rd t + \sigma_1\rd I^{1,u,h}_t
  = (\mu_{t} + u_{t})\rd t + \sigma_1\rd B^{1}_{t},\quad t\in[0,\infty),
\end{equation*}
by definition of $I^{1,u,h}$.
This justifies writing $X^u$ in place of $X^{u,h}$.

With this, the cost functional $\cJ:\cA\times\S \to[0,\infty)$ of the full information problem is defined as
\begin{align*}
	\cJ(u,h;z)=\E\Bigl[\int_{0}^{\infty}e^{-\delta t}\Bigl(
	\frac{1}{2}\big(\kappa |X^{u;z}_{t}|^{2} + \rho u_{t}^{2} \big)+ c(h_{t})\Bigr) \rd t \Bigr]
	\quad \text{subject to }\eqref{eq:state-aux}.
\end{align*}
and the value function $V:\S\to[0,\infty)$ of the {\it full information problem} is given by
\begin{equation}\label{eq:FI-value-function}
	V(z):=\inf_{(u,h)\in\cA}\cJ(u,h;z).
\end{equation}

\begin{remark}\label{rem:equivalent-problem}
The original partial information problem \eqref{eq:original-control-problem} can be embedded
into the full information problem \eqref{eq:FI-value-function} as the cost functionals are linked
via $\cJ((u,h);(x,m_{0},\gamma_{0}))=J(u,h)$ and the set of admissible controls $\cA$ is the same
for both problems.
Hence, solving the full information problem also solves the partial information problem.
\end{remark}

\section{Reduction to a Deterministic Control Problem}\label{sec:reduction}
We solve the full information problem by constructing a classical solution of the associated Hamilton--Jacobi--Bellman (HJB) equation.
The key observation in the construction of the classical solution is that the dimension of the space can be reduced to one dimension,
resulting in a deterministic control problem for the conditional variance $\gamma^h$.
The reduction is carried out in this section.

To begin with, observe that the infinitesimal generator of the state process takes the form
\begin{equation}\label{eq:differential-operator}
	\cL^{u,h}\phi
	:= (m + u)\phi_{x}
		+\lambda(\bar{\mu}-m)\phi_{m}
		+f(\gamma,h)\phi_{\gamma}
		+ \frac{\sigma_{1}^{2}}{2}\phi_{xx}
		+\frac{1}{2}\gamma^{2}(\bar{\sigma}_{1}^{2} + h)\phi_{mm}
		+\gamma \phi_{xm}\quad\text{on }\S
\end{equation}
for any $(u,h)\in\cU$ and any sufficiently regular function $\phi:\S\to\R$.
With this, the degenerate elliptic, fully nonlinear HJB equation of the full information problem reads
\begin{equation}\label{eq:HJB}\tag{$\text{HJB}_{\text{full}}$}\begin{split}
	-\delta V + \inf_{(u,h)\in\cU}\Bigl\{
	\cL^{u,h}V + \frac{1}{2}\bigl(\kappa x^{2} + \rho u^{2}\bigr) + c(h)\Bigr\} = 0\qquad\text{on }\S.
\end{split}\end{equation}
Formally minimizing with respect to the control variable $u$ leads to a candidate optimizer
\[
  u^{*}=-\frac{V_{x}}{\rho},
\]
and plugging this candidate optimizer back into the HJB equation leads to
\begin{align*}
	-\delta V + \inf_{h\in[0,\infty)}\bigg\{ -\frac{V_{x}^{2}}{2\rho} + mV_{x}
	&	+\lambda(\bar{\mu}-m)V_{m}
		+\frac{\sigma_{1}^{2}}{2}V_{xx}
		+\gamma V_{xm}+\frac{\kappa x^{2}}{2}\\
   &\qquad +f(\gamma,h)V_{\gamma} +\frac{1}{2}\gamma^{2}\bigl(\bar{\sigma}_{1}^{2} + h\bigr)V_{mm} + c(h)\bigg\} =0.
\end{align*}
This suggests to make an ansatz for the value function via the mapping $W:\S\to\R$ given by
\begin{equation}\label{eq:ansatz}
  W(x,m,\gamma) := a_{1}x^{2} + a_{2}m^{2} + a_{3}xm + b_{1}x + b_{2}m + \bar{w}(\gamma),
\end{equation}
for $a_{1},a_{2},a_{3},b_{1},b_{2}\in\R$ and a continuously differentiable function $\bar{w}:[0,\gamma_{\max}]\to[0,\infty)$ to be determined.
Plugging $W$ into the HJB equation and separating variables yields the system of equations
\begin{align*}
	& 0 = -\frac{2a_{1}^{2}}{\rho} + \frac{\kappa}{2}-\delta a_{1}, &
	& 0 = a_3 - \frac{a_{3}^{2}}{2\rho} - 2\lambda a_{2}-\delta a_{2}, &
	& 0 = 2a_{1} - \frac{2a_{1}a_{3}}{\rho} -\lambda a_{3}-\delta a_{3},\\
	& 0 = \lambda \bar{\mu} a_{3} - \frac{2a_{1}b_{1}}{\rho} -\delta b_{1}, &
	& 0 = b_1 - \frac{a_{3}b_{1}}{\rho} + 2\lambda \bar{\mu} a_{2}-\lambda b_{2}-\delta b_{2}.
\end{align*}
Here, the first three equations characterize $a_1$ to $a_3$, whereas the last two equations determine $b_1$ and $b_2$.
An explicit solution is given by
\begin{align}\label{eq:ansatz-coefficients}
	& a_{1}=-\frac{\delta \rho}{4}+\frac{\sqrt{\delta^{2} \rho^{2} + 4 \kappa \rho}}{4}, &
	& a_{2}=\frac{a_{3} (2 \rho-a_{3})}{2 \rho (2 \lambda + \delta )}, &
	& a_{3}=\frac{2 a_{1} \rho}{\delta \rho+\lambda \rho + 2 a_{1}}, \\
	& b_{1}=\frac{\lambda \bar{\mu} a_{3} \rho}{\delta \rho + 2 a_{1}}, &
	& b_{2}=\frac{2 \lambda \bar{\mu} a_{2} \rho-b_{1} a_{3}+b_{1} \rho}{\rho (\lambda + \delta)}.\notag
\end{align}
Let us point that the solution of the quadratic equation determining $a_1$ is chosen such that $a_1>0$. This guarantees in particular that $W$ is convex in $x$ as it should be.
In fact, the Hessian of $W$ with respect to $(x,m)$ is given by
\[
  A := \begin{pmatrix}2a_1 & a_3\\ a_3 & 2a_2\end{pmatrix}
\]
and since $a_1>0$ and
\[
  \det[A] = 4a_1a_2 - a_3^2 = \frac{(\delta \rho - \sqrt{\delta^2 \rho^2 + 4k\rho})^2
    \rho(\delta\rho + \sqrt{\delta^2\rho^2 + 4k\rho})}{2(\delta + 2\lambda)
    (\sqrt{\delta^2\rho^2 + 4k\rho} + (\delta + 2\lambda)\rho)^2} > 0,
\]
it follows that $A$ is positive definite and hence $W$ is strictly convex in $(x,m)$.
With the constants $a_1,a_2,a_3,b_1,b_2$ fixed and upon setting $C_1 := \sigma_{1}^{2}a_{1} + \lambda \bar{\mu} b_{2} -b_{1}^{2}/({2\rho})$, the HJB equation reduces to
\begin{equation}\label{eq:HJB-reduced-problem-unreduced}
	-\delta\bar{w} + \inf_{h\in[0,\infty)}\Bigl\{
	f(\gamma,h)\bar{w}' + C_1 + a_{3} \gamma
	+(\bar{\sigma}_{1}^{2} + h)a_{2}\gamma^{2}
	+c(h)\Bigr\}=0\qquad\text{on }[0,\gamma_{\max}],
\end{equation}
which is now formally a nonlinear first-order ordinary differential equation for $\bar{w}$.
Next, we propose another ansatz for $\bar{w}$ given by
\begin{equation}\label{eq:ansatz-w}
 \bar{w}(\gamma) = w(\gamma) + a_{2}\gamma
	+ \frac{1}{\delta}\bigl(C_{1} + a_{2}\sigma_{2}^{2}\bigr),\quad \gamma\in[0,\gamma_{\max}],
\end{equation}
for another continuously differentiable function $w:[0,\gamma_{\max}]\to[0,\infty)$ which
is to be determined.
Using this ansatz and setting
\begin{equation}\label{eq:relation-a2-a3}
	\bar{a} := a_{3} - a_{2}(2\lambda + \delta) = \frac{a_{3}^{2}}{2\rho} > 0,
\end{equation}
the HJB equation reduces further to
\begin{equation}\label{eq:HJB-reduced-problem}\tag{$\text{HJB}_{\text{det}}$}
	- \delta w + \inf_{h\in[0,\infty)}\bigl\{
	f(\gamma,h)w' + \bar{a}\gamma + c(h)\bigr\} = 0\quad\text{on }[0,\gamma_{\max}]
\end{equation}
to be solved for $w$.
Observe that \eqref{eq:HJB-reduced-problem} is a Hamilton--Jacobi--Bellman equation
connected to a deterministic control problem with the controlled state being
the conditional variance $\gamma^h$.
Using this control interpretation, we proceed to construct a classical solution
of \eqref{eq:HJB-reduced-problem} which, in light of \eqref{eq:ansatz} and \eqref{eq:ansatz-w},
induces a classical solution of the HJB equation \eqref{eq:HJB} of the full information problem.

We proceed by formulating the deterministic control problem.
For this, we first recall that the cost function $c$ is assumed to be strictly increasing and strictly convex (see Assumption \ref{ass:cost}).
This implies, in particular, that the derivative $c'$ is strictly increasing and hence invertible.
With this, it follows that the positive constants
\begin{equation}\label{eq:constants}
	L_{v} := \frac{\bar{a}}{\delta}, \qquad
	M_{0} := \gamma_{\max}^{2} L_{v}, \qquad
	h_{\max} :=(c')^{-1}(M_{0})
\end{equation}
are well-defined.
The constant $L_v$ will turn out to be a Lipschitz constant for the value function of the deterministic control problem,
whereas $h_{\max}$ will serve as an upper bound on the optimal information acquisition rate.
In light of this and since \eqref{eq:HJB-reduced-problem} suggests that the control problem is
deterministic, we subsequently work with a reduced set of admissible information acquisition rates given by
\[
    \cA^{obs}_{det}:=\Bigl\{h:[0,\infty)\to[0,h_{\max}] \;\Big|\;
	h \text{ is Borel-measurable}\Bigr\}.
\]
Observe that Lemma~\ref{lem:gamma-bounded} is still applicable and hence the state equation
\begin{equation}\label{eq:SEE}
  \rd \gamma_t^{h} = f(\gamma_t^{h},h_t)\rd t,
  \quad \gamma_0^{h} = \gamma,\quad t\in[0,\infty),
\end{equation}
admits an absolutely continuous solution $\gamma^{h}=\gamma^{h;\gamma}$ for all $h\in\cA^{obs}_{det}$
and $\gamma\in[0,\gamma_{\max}]$.
Next, we introduce the auxiliary cost function $k:[0,\gamma_{\max}]\times[0,h_{\max}]\to[0,\infty)$ given by
\[
  k(\gamma,h) := \bar{a}\gamma + c(h),
\]
and, with this, consider the cost functional $\cJ_{det}:\cA^{obs}_{det}\times[0,\gamma_{\max}]\to[0,\infty)$ with
\[
	\cJ_{det}(h;\gamma):= \int_{0}^{\infty}
	e^{-\delta t}k(\gamma^{h;\gamma}_{t},h_{t})\rd t\quad \text{subject to }\eqref{eq:SEE}.\\
\]
The value function $v:[0,\gamma_{\max}]\to[0,\infty)$ of this {\it reduced problem} is defined as
\begin{equation}\label{eq:def-reduced-value-function}
	v(\gamma):= \inf_{h\in\cA^{obs}_{det}}\cJ_{det}(h;\gamma).
\end{equation}
Hence, we subsequently consider a nonlinear deterministic control problem on the compact state space 
$[0,\gamma_{\max}]$ with compact control set $[0,h_{\max}]$ and regular coefficient and cost functions.

\section{Solution of the Reduced and Full Information Problem}\label{sec:reduced-control-problem}
In this section, we solve both the reduced and full information control problem.
We first establish existence of an optimal control for the reduced problem
and then study the properties of the value function
which eventually allow us to conclude that $v$ defined in \eqref{eq:def-reduced-value-function} is continuously differentiable.
This regularity leads us to define a candidate feedback map
which we prove to induce an optimal control in a verification theorem.

\subsection{Existence of Optimal Controls and Value Function Properties}
The existence of an optimal control considerably simplifies the regularity analysis.
As we have restricted the controls to take values in a compact interval and the cost function is sufficiently regular,
we can apply a classical result to obtain existence of an optimal control.

\begin{theorem}\label{thm:existence}
    For every initial condition $\gamma\in[0,\gamma_{\max}]$,
    there exists an optimal continuous control $h^{*}\in\cA^{obs}_{det}$ for $v(\gamma)$.
\end{theorem}
\begin{proof}
    We check the assumptions of \cite[Corollary VI.1.4]{bardi1997optimal}.
    First, $f$ and $k$ are continuous, and the control set $[0,h_{\max}]$ is compact.
    The Lipschitz continuity of $f$ and $k$ in space uniformly in the control variable is immediate from 
    the differentiability of $f,k$ and the compactness of $[0,\gamma_{\max}]\times[0,h_{\max}]$.
    This also implies the boundedness of $k$.
    THe positivity of the discount rate $\delta$ is by definition.
    Next, $h\mapsto f(\gamma,h),k(\gamma,h)$ are continuous for all $\gamma\in[0,\gamma_{\max}]$ and thus both $f(\gamma,[0,h_{\max}])$ and $k(\gamma,[0,h_{\max}])$ are intervals.
    This shows that the set
    $f(\gamma,[0,h_{\max}]) \times k(\gamma,[0,h_{\max}])$
    is convex for all $\gamma\in[0,\gamma_{\max}]$.
    With this, the existence of an optimal control $h^*$ follows from 
    \cite[Corollary VI.1.4]{bardi1997optimal}.
    Finally, the continuity of $h^*$ is a consequence of the regularity of the
    Hamiltonian $(\gamma,h,p)\mapsto f(\gamma,h)p + k(\gamma,h)$ implied by the
    strict convexity and the twice continuous differentiability of $c$;
    see \cite[Definition 3.19, Proposition 3.22]{grass2008optimal}.
\end{proof}

Having established the existence of an optimal control,
we can now proceed to study several properties of the value function $v$.
In particular, we show that $v$ is strictly concave, strictly increasing,
and continuously differentiable.
The main tool in proving these results is the observation that the conditional variance is at least 
twice continuously differentiable with respect its initial value, as shown in the following lemma.

\begin{lemma}\label{lem:regularity-state}
Let $h\in\cA^{obs}_{det}$ be continuous and $\gamma_0\in[0,\gamma_{\max}]$.
Then
\[
 \beta_t := \frac{\partial}{\partial \gamma_0}\gamma^{h;\gamma_0}_t
 \quad\text{and}\quad
 \eta_t := \frac{\partial^2}{\partial \gamma_0^2}\gamma^{h;\gamma_0}_t
\]
exist for all $t\in[0,\infty)$.
Moreover, $\beta$ is strictly positive and bounded independently of $(t,\gamma_0,h)$,
whereas $\eta$ is strictly negative for $t>0$ and satisfies the lower bound
\[
  \eta_t \geq -\frac{R_1}{R_2}\Bigl(1 - e^{-R_2t}\Bigr),\quad t\in[0,\infty),
\]
with $R_1:=2(\bar\sigma_1^2+h_{\max})$ and $R_2:=2(\bar\sigma_1^2+h_{\max})\gamma_{\max} + \lambda$.
\end{lemma}

\begin{proof}
The existence of $\beta$ and $\eta$ is a classical result; see e.g.\ \cite[Theorem V.3.1]{hartman2002ordinary}.
In fact, we have
\begin{equation}\begin{split}\label{eq:beta-ODE}
    \rd \beta_{t}
    =  f_{\gamma}(\gamma^{h;\gamma_0}_{t},h_{t})\beta_{t}\rd t
    =  -2\bigl((\bar{\sigma}_{1}^{2} + h_{t})\gamma^{h;\gamma_0}_{t}
    +\lambda\bigr)\beta_{t}\rd t,
    \quad \beta_{0}=1, \quad t\geq 0.
\end{split}\end{equation}
and
\begin{equation*}
    \rd \eta_{t}
    = \bigl[f_{\gamma\gamma}(\gamma^{h;\gamma_0}_{t},h_{t})\beta_{t}^{2}
    +f_{\gamma}(\gamma^{h;\gamma_0}_{t},h_{t})\eta_{t}\bigr]\rd t
    = -2\bigl[(\bar{\sigma}_{1}^{2} + h_{t})(\beta_{t}^{2} + \gamma^{h;\gamma_0}_t\eta_t)
    + \lambda\eta_{t}\bigr]\rd t, \quad
    \eta_{0} = 0, \quad t\geq 0.
\end{equation*}
Observe that the coefficient of $\beta_{t}$ in the ODE \eqref{eq:beta-ODE} is strictly negative
and bounded in $(\gamma_t,h_t)$ on $[0,\gamma_{\max}]\times[0,h_{\max}]$.
As a result, $\beta$ is $(0,1]$-valued, hence strictly positive and uniformly bounded.
Finally, since the coefficient in the ODE for $\eta$ is strictly negative,
it follows that $\eta_t$ is strictly negative for all $t>0$.
Moreover, we have
\[
  -2\bigl[(\bar{\sigma}_{1}^{2} + h_{t})(\beta_{t}^{2} + \gamma^{h;\gamma_0}_t\eta_t)
    + \lambda\eta_{t}\bigr] \geq -(R_1 + R_2\eta_t),\quad t\geq 0,
\]
from which we conclude by a standard comparison result that $\eta_t\geq\rho_t$, where $\rho_t$ solves
\[
  \rd \rho_t = -(R_1+R_2\rho_t)\rd t,
  \quad \rho_0 = 0,\quad t\geq 0,
\]
and is explicitly given by $\rho_t=-R_1(1 - e^{-R_2t})/R_2$.
\end{proof}
The combination of the differentiability of the states with respect to the initial condition
and the existence of an optimal continuous control
allows us to infer several properties of the value function.

\begin{theorem}\label{thm:regularity-value-function}
The value function $v$ of the reduced problem is strictly increasing, strictly concave,
and Lipschitz continuous with constant $L_v$.
\end{theorem}

\begin{proof}
Fix $\gamma_0\in[0,\gamma_{\max}]$ and let $h\in\cA^{obs}_{det}$ be an arbitrary continuous control.
Just as in Lemma~\ref{lem:regularity-state}, denote by $\beta$ and $\eta$ the first- and second-order
derivatives of $\gamma^h:=\gamma^{h;\gamma_0}$ with respect to $\gamma_0$.
Using the Leibniz rule, see e.g.\ \cite[8.11.2]{dieudonne2011foundations}, we compute
\begin{align}\label{eq:J-first-derivative}
    \frac{\rd}{\rd \gamma_0}\cJ_{det}(h;\gamma_0)
    =\int_{0}^{\infty}e^{-\delta t}k_{\gamma}(\gamma_{t}^{h},h_{t})\beta_{t}\rd t
    =\bar{a}\int_{0}^{\infty} e^{-\delta t}\beta_{t}\rd t > 0,
\end{align}
where we have used that $\beta$ is strictly positive.
In particular, $\gamma_0\mapsto\cJ_{det}(h;\gamma_0)$ is strictly increasing.
Letting $\gamma_0,\bar\gamma_0\in[0,\gamma_{\max}], \gamma_0<\bar\gamma_0$, and assuming that $h$ is an
optimal continuous control for $\bar\gamma_0$, it follows that
\begin{align*}
    v(\bar\gamma_0)
    =\cJ_{det}(h;\bar\gamma_0)
    > \cJ_{det}(h;\gamma_0)
    \geq v(\gamma_0),
\end{align*}
and thus the value function $v$ is also strictly increasing.

Next, since $\beta$ is upper bounded by $1$, it follows that
\[
\frac{\rd}{\rd \gamma_0}\cJ_{det}(h;\gamma_0) = \bar{a}\int_{0}^{\infty} e^{-\delta t}\beta_{t}\rd t
\leq \bar{a}\int_{0}^{\infty} e^{-\delta t}\rd t = \frac{\bar{a}}{\delta} = L_v.
\]
Hence $\gamma_0\mapsto\cJ_{det}(h;\gamma_0)$ is Lipschitz continuous, uniformly in $h$.
But this implies that the value function $v$ is Lipschitz continuous with the same Lipschitz constant $L_v$ by \cite[Proposition 1.32]{weaver2018lipschitz}.

Finally, regarding the strict concavity of $v$, we first compute
\begin{equation}\label{eq:J-second-derivative}
    \frac{\rd^{2}}{\rd \gamma_0^{2}}\cJ_{det}(h;\gamma_0)
    =\bar{a}\int_{0}^{\infty} e^{-\delta t}\eta_{t}\rd t < 0,
\end{equation}
where we have used that $\eta_t<0$ for all $t>0$.
Hence $\gamma_0\mapsto\cJ_{det}(h;\gamma_0)$ is strictly concave.
Now let $\alpha\in(0,1)$ and define the convex combination
$\gamma^{\alpha}_0:=\alpha\gamma_0 + (1-\alpha)\bar\gamma_0$.
Let $h^{\alpha}$ be an optimal continuous control for $\gamma^{\alpha}_0$ and observe
that, by the strict concavity of $\gamma_0\mapsto\cJ_{det}(h^\alpha;\gamma_0)$,
\begin{align*}
    v(\gamma^{\alpha}_0)
    =\cJ_{det}(h^{\alpha};\gamma^{\alpha}_0)
    > \alpha \cJ_{det}(h^{\alpha};\gamma_0) + (1-\alpha) \cJ_{det}(h^{\alpha};\bar\gamma_0)
    \geq \alpha v(\gamma_0) + (1-\alpha) v(\bar\gamma_0),
\end{align*}
so the value function is strictly concave as claimed.
\end{proof}

\begin{remark}
\label{rmk:concavity}
Since the conditional variance is a measure for the uncertainty present in the partial information problem
(which we argue is even {\it the} measure for information in this problem),
the concavity of the value function $v$ means that increasing uncertainty
comes with an increase of the \emph{marginal} cost.
This finding is in line with the findings of \cite{radner1984nonconcavity},
and stronger for out specific setting.
There, non-concavity of the value function on a finite state space for a information structure maximization problem was observed (corresponding to non-convexity in a minimization problem).
\end{remark}

\subsection{Conditional Precision and Value Function Regularity}
Our next aim is to show that the value function $v$ of the reduced problem is continuously differentiable.
For this, it turns out to be convenient to reparametrize the problem in terms of the
{\it conditional precision}, i.e.\ the inverse of the conditional variance.
In what follows, whenever $x$ is a quantity parameterized in terms of conditional variance,
we denote its reparametrization in terms of conditional precision by $\check{x}$.
More precisely, for any initial variance $\gamma_0\in[0,\gamma_{\max}]$ and any control
$h\in\cA^{obs}_{det}$, we define
\footnote{We use the convention $1/0:=\infty$ here.
Moreover, we subsequently agree upon the convention $1/\infty=0$.}
the conditional precision as $\check{\gamma}^{h;\check{\gamma}_0}:=1/\gamma^{h;\gamma_0}$, where $\check{\gamma}_0:= 1/\gamma_0$ denotes the initial precision.
Writing $\check{\gamma}_{\min}:= 1/\gamma_{\max}>0$ for the minimum precision, we see that $\check{\gamma}^{h;\check{\gamma}_0}$ is $[\check{\gamma}_{\min},\infty]$-valued.
Moreover, a straightforward change of variables shows that $\check{\gamma}^{h;\check{\gamma}_0}_{t}$ is an absolutely continuous function satisfying
\begin{equation}\label{eq:SEE-reparameterized}
  \rd \check{\gamma}^{h;\check{\gamma}_0}_{t} = \check{f}(\check{\gamma}^{h;\check{\gamma}_0}_{t},h_{t})\rd t,
    \quad \check{\gamma}^{h;\check{\gamma}_0}_{0}=\check{\gamma}_0,
	\quad t\in[0,\infty),
\end{equation}
almost everywhere, where
\begin{equation*}
  \check{f}(\check{\gamma},h)
    := \bar{\sigma}_{1}^{2} + h + 2\lambda\check{\gamma} - \sigma_{2}^{2}\check{\gamma}^{2},
    \quad (\check{\gamma},h)\in [\check{\gamma}_{\min},\infty]\times[0,h_{\max}].
\end{equation*}
Introducing the reparametrized cost function
$\check{k}:[\check{\gamma}_{\min},\infty]\times[0,h_{\max}]\to[0,\infty)$ as 
\begin{equation*}
  \check{k}(\check{\gamma},h)
    := k\bigl(1/\check{\gamma},h\bigr)
	=\frac{\bar{a}}{\check{\gamma}} + c(h),
\end{equation*}
the reparametrized cost functional $\check{\cJ}_{det}: \cA^{obs}_{det} \times [\check{\gamma}_{\min},\infty] \to [0,\infty)$ and the reparametrized value function $\check{v}: [\check{\gamma}_{\min},\infty]\to[0,\infty)$ can be written as
\begin{equation}\label{eq:reparameterized-control-problem}
  \check{\cJ}_{det}(h;\check{\gamma}_0) := \int_{0}^{\infty}e^{-\delta t}
    \check{k}(\check{\gamma}^{h;\check{\gamma}_0}_{t}, h_{t})\rd t
  \quad\text{and}\quad
  \check{v}(\check{\gamma}_0) := \inf_{h\in\cA^{obs}_{det}}\check{\cJ}_{det}(h;\check{\gamma}_0).
\end{equation}
It is clear that $\cJ_{det}(h;\gamma)=\check{\cJ}_{det}(h;1/\gamma)$ for any $h\in\cA^{obs}_{det}$ and hence $v(\gamma)=\check{v}(1/\gamma)$.
In particular, a control is optimal for $v(\gamma)$ if and only if it is also optimal for $\check{v}(1/\gamma)$.

We proceed to show that $\check{v}$ is convex and semiconcave, and hence continuously differentiable.
From this, we conclude that $v$ must also be continuously differentiable.
As previously, the idea is to use that the state process is sufficiently regular,
which translates into regularity of the cost functional and the value function.

To formulate the regularity result, let us first fix some notation.
We write $\cC^0$ for the space of bounded continuous functions $h:[0,\infty)\to\R$ and equip this space with the norm
\[
  \|h\|_{\infty} := \sup_{t\in[0,\infty)}|h(t)|,\qquad h\in\cC^0.
\]
With this, we set $\cS:=\cC^0\times\R$ with norm
\[
  \|(h,\check{\gamma}_0)\|_{\cS} := \|h\|_{\infty} + |\check{\gamma}_{0}|,\qquad (h,\check{\gamma}_0)\in\cS.
\]
Given another normed space $(Y,\|\argdot\|_Y)$ and an operator $L:\cS\to Y$, we write
\begin{equation}\begin{aligned}\label{eq:direc-derivative}
  \rD L(h,\check{\gamma}_0)[h',\check{\gamma}_0'] &:= \frac{\rd}{\rd\varepsilon} L\bigl(h+\varepsilon h',\check{\gamma}_0+\varepsilon\check{\gamma}_0'\bigr)\Big|_{\varepsilon=0},\\
  \rD^2 L(h,\check{\gamma}_0)[h',\check{\gamma}_0'] &:= \frac{\rd^2}{\rd\varepsilon^2} L\bigl(h+\varepsilon h',\check{\gamma}_0+\varepsilon\check{\gamma}_0'\bigr)\Big|_{\varepsilon=0},
\end{aligned}\end{equation}
for the first- and second-order directional derivatives of $L$ at $(h,\check{\gamma}_0)$ in direction $(h',\check{\gamma}_0')$.

\begin{lemma}\label{lem:state_diffable}
Let $\nu\in(0,\bar{\sigma}_1^2)$.
Then the mapping
$(h,\check{\gamma}_0)\mapsto \check{\gamma}^{h;\check{\gamma}_0}_{t}$
is twice differentiable on the open set
\[
  \cS_\nu := \bigl\{h \in \cC^0 : -\bar{\sigma}_1^2 + \nu < h < h_{\max}+\nu\bigr\}
  \times (\check{\gamma}_{\min}-\nu,\infty) \subset \cS
\]
in any direction $(h',\check{\gamma_0}')\in\cS$.
Moreover, for any $(h;\check{\gamma}_0)\in\cS_\nu$, the mapping
$(h',\check{\gamma}_0')\mapsto \rD \check{\gamma}^{h;\check{\gamma}_0}_{t}[h',\check{\gamma}_0']$
is linear and bounded on $\cS$ and satisfies
$\rD \check{\gamma}_t^{h;\check{\gamma}_0}[h',\check{\gamma}_0']>0$ for all
$t\in[0,\infty)$, whereas
$\rD^2 \check{\gamma}_t^{h;\check{\gamma}_0}[h',\check{\gamma}_0']\leq 0$ for
all $t\in[0,\infty)$.
\end{lemma}

\begin{proof}
First, note that by Remark~\ref{rem:negative-h} the conditional precision
$\check{\gamma}^{h;\check{\gamma}_0}$ is well-defined for all $(h,\check{\gamma}_0)\in\cS_\nu$,
even if $h$ becomes negative.
Using the chain rule twice, we find that
\begin{align*}
  \rD \check{\gamma}_t^{h;\check{\gamma}_0}[h',\check{\gamma}_0']
  = \frac{\rd}{\rd\varepsilon} \check{\gamma}_t^{h+\varepsilon h';\check{\gamma}_0+\varepsilon\check{\gamma}_0'}\Big|_{\varepsilon=0}
  &= \frac{\rd}{\rd\varepsilon}\check{\gamma}^{h+\varepsilon h';\check{\gamma}_0}_{t}\Big|_{\varepsilon=0}
    + \frac{\rd}{\rd\varepsilon}\check{\gamma}^{h;\check{\gamma}_0+\varepsilon \check{\gamma}_0'}_{t}\Big|_{\varepsilon=0}\\
  &= \frac{\rd}{\rd\varepsilon}\check{\gamma}^{h+\varepsilon h';\check{\gamma}_0}_{t}\Big|_{\varepsilon=0}
    + \check{\gamma}_0'\frac{\partial}{\partial\check{\gamma}_0}\check{\gamma}^{h;\check{\gamma}_0}_{t}
  =: \beta^{h}_t[h'] + \check{\gamma}_0'\beta^{\check{\gamma}_0}_t,
\end{align*}
and, as in the proof of Lemma~\ref{lem:regularity-state}, it follows from \cite[Theorem V.3.1]{hartman2002ordinary} that $\beta^{h}[h']$ and $\beta^{\check{\gamma}_0}$ exist and follow the dynamics
\begin{align*}
  \rd \beta^{\check{\gamma}_0}_{t}
    &= 2\bigl(\lambda - \sigma_{2}^{2}\check{\gamma}_{t}^{h;\check{\gamma}_0}\bigr)\beta^{\check{\gamma}_0}_{t}\rd t,
	& \beta^{\check{\gamma}_0}_{0} &= 1, & t&\in[0,\infty),\\
  \rd \beta^{h}_{t}[h']
    &= 2(\lambda - \sigma_{2}^{2}\check{\gamma}_{t}^{h;\check{\gamma}_0})\beta^{h}_{t}[h']\rd t + {h}_{t}'\rd t,
	& \beta^{h}_{0}[h'] &= 0, & t&\in[0,\infty).
\end{align*}
These differential equations can be solved explicitly yielding
\[
  \beta^{\check{\gamma}_0}_{t} = \exp\Bigl\{2\int_0^t \bigl(\lambda - \sigma_{2}^{2}\check{\gamma}_{s}^{h;\check{\gamma}_0}\bigr) \rd s\Bigr\}
  \quad\text{and}\quad
  \beta^{h}_{t}[h'] = \int_0^t h'_s \beta^{\check{\gamma}_0}_{t-s} \rd s,
  \qquad t\in[0,\infty).
\]
From this, we immediately find that $\rD \check{\gamma}^{h;\check{\gamma}_0}_{t}$ is linear,
bounded, and takes strictly positive values.

The argument for the second-order derivative is similar.
Using the chain rule, it follows that it takes the form
\[
  \rD^2\check{\gamma}_t^{h;\check{\gamma}_0}[h',\check{\gamma}_0']
  = \frac{\rd^2}{\rd\varepsilon^2} \check{\gamma}_t^{h+\varepsilon h';\check{\gamma}_0+\varepsilon\check{\gamma}_0'}\Big|_{\varepsilon=0}
  = \eta^{(h,h)}_t[h'] + 2\check{\gamma}_0'\eta_t^{h,\check{\gamma_0}}[h'] + (\check{\gamma}_0')^2\eta^{(\check{\gamma}_0,\check{\gamma}_0)}_t,
\]
where the derivatives $\eta^{(h,h)}[h']$, $\eta^{(h,\check{\gamma}_0)}[h']$, and $\eta^{(\check{\gamma}_0,\check{\gamma}_0)}$ satisfy
\begin{align*}
  \rd \eta^{(h,h)}_{t}[h']
    &= 2\bigl[(\lambda - \sigma_{2}^{2}\check{\gamma}_{t}^{h;\check{\gamma}_0})\eta^{(h,h)}_{t}[h'] - \sigma_{2}^{2}(\beta^{h}_{t}[h'])^{2}\bigr]\rd t,
    & \eta^{(h,h)}_{0}[h'] &= 0,
    & t&\in[0,\infty),\\
  \rd \eta^{(h,\check{\gamma}_0)}_{t}[h']
    &= 2\bigl[(\lambda - \sigma_{2}^{2}\check{\gamma}_{t}^{h;\check{\gamma}_0})\eta^{(h,\check{\gamma}_0)}_{t}[h'] - \sigma_{2}^{2}\beta_{t}^{h}[h']\beta^{\check{\gamma}_0}_t\bigr]\rd t,
    & \eta^{(h,\check{\gamma}_0)}_{0}[h'] &= 0,
    & t&\in[0,\infty),\\
  \rd\eta^{(\check{\gamma}_0,\check{\gamma}_0)}_{t}
    &= 2\bigl[(\lambda - \sigma_{2}^{2}\check{\gamma}_{t}^{h;\check{\gamma}_0})\eta^{(\check{\gamma}_0,\check{\gamma}_0)}_{t} - \sigma_{2}^{2}(\beta_{t}^{\check{\gamma}_0})^{2}\bigr]\rd t,
	& \eta^{(\check{\gamma}_0,\check{\gamma}_0)}_{0} &= 0,
    & t&\in[0,\infty).
\end{align*}
Once again, these equations can be solved explicitly, yielding
\begin{align*}
  \eta^{(h,h)}_{t}[h'] &= - 2\sigma_{2}^{2}\int_0^t (\beta^{h}_{t}[h'])^{2}\beta^{\check{\gamma}_0}_{t-s} \rd s,\\
  \eta^{(h,\check{\gamma}_0)}_{t}[h'] &= - 2\sigma_{2}^{2}\int_0^t \beta_{t}^{h}[h']\beta^{\check{\gamma}_0}_t \beta^{\check{\gamma}_0}_{t-s} \rd s,\\
  \eta^{(\check{\gamma}_0,\check{\gamma}_0)}_{t} &= - 2\sigma_{2}^{2}\int_0^t (\beta_{t}^{\check{\gamma}_0})^{2} \beta^{\check{\gamma}_0}_{t-s} \rd s,
\end{align*}
for all $t\in[0,\infty)$. Since $\beta^h[h'], \beta^{\check{\gamma}_0}\geq 0$, we conclude that $\rD^2 \check{\gamma}_t^{h;\check{\gamma}_0}[h',\check{\gamma}_0']\leq 0$ for all $t\in[0,\infty)$.
\end{proof}

With the second order directional differentiability of the state process at hand,
it is straighforward to show that the reparametrized cost functional is strictly convex.

\begin{proposition}
The cost functional $\check{\cJ}_{det}$ is strictly convex on
$(\cA^{obs}_{det}\times[\check{\gamma}_{\min},\infty))\cap\cS$.
\end{proposition}

\begin{proof}
Let $\nu\in(0,\bar{\sigma}_1^2)$ and $\cS_\nu$ as in Lemma~\ref{lem:state_diffable}.
We begin by showing that for any $(h,\check{\gamma}_0)\in\cS_\nu$, the cost
functional $\check{\cJ}_{det}(h;\check{\gamma}_0)$ is
twice differentiable in any direction $(h',\check{\gamma}_0')\in\cS$.
Moreover, we show that the first-order derivative
$\rD\check{\cJ}_{det}(h;\check{\gamma}_0)[h',\check{\gamma}_0']$ is linear and
continuous in $(h',\check{\gamma}_0')\in\cS$.
Note that we formally defined the cost function $\check{k}$ only on
$[\check{\gamma}_{\min},\infty]\times[0,h_{\max}]$, so that we need to
extend $\check{k}$ to a slightly larger domain to make all quantities well-defined.
The same applies to $\check{\cJ}_{det}$, which we only defined on
$\cA^{obs}_{det}\times[\check{\gamma}_{\min},\infty]$, but which can be extended
to $\cS_\nu$ in an obvious way.
Regarding the differentiability of $\check{\cJ}_{det}$,
we use the Leibniz rule \cite[8.11.2]{dieudonne2011foundations} and
Lemma~\ref{lem:state_diffable} to compute
\begin{align*}
  \rD \check{\cJ}_{det}(h;\check{\gamma}_0)[h',\check{\gamma}_0']
  &= \frac{\rd}{\rd\varepsilon}\check{\cJ}_{det}\bigl(h+\varepsilon h';\check{\gamma}_0 + \varepsilon\check{\gamma}_0'\bigr)\Big|_{\varepsilon=0}\\
  &= \int_0^\infty e^{-\delta t}\frac{\rd}{\rd\varepsilon}\check{k}\bigl(\check{\gamma}_t^{h+\varepsilon h';\check{\gamma}_0+\varepsilon\check{\gamma}_0'},h_t+\varepsilon h_t'\bigr)\Big|_{\varepsilon=0}\rd t\\
  &= -\int_0^\infty e^{-\delta t}\Bigl(\bar{a}\bigl(\check{\gamma}_t^{h;\check{\gamma}_0}\bigr)^{-2}\rD\check{\gamma}_t^{h;\check{\gamma}_0}[h',\check{\gamma}_0'] + c'(h_t)h'_t\Bigr) \rd t.
\end{align*}
Since $(h',\check{\gamma}_0')\mapsto\rD\check{\gamma}_t^{h;\check{\gamma}_0}[h',\check{\gamma}_0']$
is linear by Lemma~\ref{lem:state_diffable}, we see that also
$(h',\check{\gamma}_0')\mapsto\rD \check{\cJ}_{det}(h;\check{\gamma}_0)[h',\check{\gamma}_0']$ is linear.
Regarding the continuity, we use that
\[
  \bigl\|\rD\check{\gamma}^{h;\check{\gamma}_0}[h',\check{\gamma}_0']\bigr\|_\infty\leq M \|(h',\check{\gamma}_0')\|_{\cS},\qquad (h,\check{\gamma}_0),(h',\check{\gamma}_0')\in \bigl(\cA^{obs}_{det}\times[\check{\gamma}_{\min},\infty)\bigr)\cap\cS,
\]
for a constant $M>0$ by Lemma~\ref{lem:state_diffable} and that $c'$ is non-decreasing to estimate
\begin{align*}
  \bigl|\rD \check{\cJ}_{det}(h;\check{\gamma}_0)[h',\check{\gamma}_0']\bigr|
  &\leq \int_0^\infty e^{-\delta t} \Bigl(\bar{a}M\bigl(\check{\gamma}_{\min}-\nu\bigr)^{-2}\|(h',\check{\gamma}_0')\|_{\cS} + c'(h_{\max}+\nu)|h'_t|\Bigr) \rd t\\
  &\leq \frac{\bar{a}(\check{\gamma}_{\min}-\nu)^{-2}M + c'(h_{\max}+\nu)}{\delta}\|(h',\check{\gamma}_0')\|_{\cS},
\end{align*}
implying that $(h',\check{\gamma}_0')\mapsto\rD \check{\cJ}_{det}(h;\check{\gamma}_0)[h',\check{\gamma}_0']$ is a linear and bounded, and hence continuous, operator on $\cS$.
For the second-order derivative, we find that
\begin{align*}
  \rD^2 \check{\cJ}_{det}(h;\check{\gamma}_0\bigr)[h',\check{\gamma}_0']
  &= \frac{\rd^2}{\rd\varepsilon^2}\check{\cJ}_{det}\bigl(h+\varepsilon h';\check{\gamma}_0 + \varepsilon\check{\gamma}_0'\bigr)\Big|_{\varepsilon=0}\\
  &= \int_0^\infty e^{-\delta t}\frac{\rd^2}{\rd\varepsilon^2}\check{k}\bigl(\check{\gamma}_t^{h+\varepsilon h';\check{\gamma}_0+\varepsilon\check{\gamma}_0'},h_t+\varepsilon h_t'\bigr)\Big|_{\varepsilon=0}\rd t\\
  &= 2\bar{a}\int_0^\infty e^{-\delta t}\bigl(\check{\gamma}_t^{h+\varepsilon h';\check{\gamma}_0+\varepsilon\check{\gamma}_0'}\bigr)^{-3}\bigl(\rD \check{\gamma}_t^{h;\check{\gamma}_0}[h',\check{\gamma}_0']\bigr)^2\rd t\\
  &\qquad - \bar{a}\int_0^\infty e^{-\delta t}\bigl(\check{\gamma}_t^{h+\varepsilon h';\check{\gamma}_0+\varepsilon\check{\gamma}_0'}\bigr)^{-2}\rD^2 \check{\gamma}_t^{h;\check{\gamma}_0}[h',\check{\gamma}_0']\rd t  + \int_0^\infty e^{-\delta t}c''(h_t)(h'_t)^2\rd t.
\end{align*}
Since $\rD \check{\gamma}^{h;\check{\gamma}_0}[h',\check{\gamma}_0']>0$, $\rD^2 \check{\gamma}^{h;\check{\gamma}_0}[h',\check{\gamma}_0']\leq 0$ and $c''\geq 0$, it follows that $\rD^2 \check{\cJ}_{det}(h;\check{\gamma}_0)[h',\check{\gamma}_0']> 0$.

Let us now turn to the proof of convexity.
For this, we fix
$z^i := (h^i,\check{\gamma}_0^i)\in (\cA^{obs}_{det}\times[\check{\gamma}_{\min},\infty))\cap\cS$
for $i=1,2$.
Note that there exists $\Theta>0$ such that $z^1+\theta z^2\in\cS_\nu$
for all $\theta\in[0,\Theta]$.
As such, the function $f:[0,\Theta]\to\R$ given by
\[
  f(\theta) := \check{\cJ}_{det}(z^{1} + \theta z^{2}),\quad \theta\in[0,\Theta],
\]
is well-defined.
Moreover, we observe that for any $\theta\in(0,\Theta)$ we have
\[
  f'(\theta)
  = \lim_{\varepsilon\to 0}\frac{f(\theta+\varepsilon) - f(\theta)}{\varepsilon}
  = \lim_{\varepsilon\to 0}\frac{\check{\cJ}_{det}(z^{1} + (\theta+\varepsilon)z^{2})
    - \check{\cJ}_{det}(z^{1} + \theta z^{2})}{\varepsilon}
  = \rD \check{\cJ}_{det}(z^{1} + \theta z^{2})[z^{2}],
\]
and similarly for the second-order derivative
\[
  f''(\theta) = \rD^{2}\check{\cJ}_{det}(z^{1} + \theta z^{2})[z^{2}].
\]
A Taylor expansion of $f$ now yields the existence of $\theta_0\in[0,\Theta]$
such that
\begin{align*}
  \check{\cJ}_{det}(z^{1} + \Theta z^{2})
  = f(\Theta)
  &= f(0) + f'(0)\Theta + \frac{1}{2}f''(\theta_{0})\Theta^2\\
  &= \check{\cJ}_{det}(z^{1})
	+ \rD \check{\cJ}_{det}(z^{1})[z^{2}]\Theta
	+ \frac{1}{2} \rD^{2} \check{\cJ}_{det}(z^{1} + \theta_{0}z^{2})[\theta_{0}z^{2}]\Theta^2.
\end{align*}
Using $\rD^{2} \check{\cJ}_{det}(z^{1} + \theta_{0}z^{2})[\theta_{0}z^{2}]>0$, we obtain
\begin{equation}\label{eq:strict-convexity-det}
  \check{\cJ}_{det}(z^{2} + \Theta z^{1})
	- \check{\cJ}_{det}(z^{1})
	> \rD \check{\cJ}_{det}(z^{1})[z^{2}]\Theta.
\end{equation}
Let now $\bar{z}^{i}\in(\cA^{obs}_{det}\times[\check{\gamma}_{\min},\infty))\cap\cS$,
$\alpha\in(0,1)$ and set $\bar{z}^{\alpha} := \alpha\bar{z}^{2} + (1-\alpha)\bar{z}^{1}$.
We evaluate \eqref{eq:strict-convexity-det} with
\[
  (z^{1},z^{2}) = \bigl(\bar{z}^{\alpha},\alpha(\bar{z}^{1} - \bar{z}^{2})\bigr)
  \quad\text{and}\quad
  (z^{1},z^{2}) = \bigl(\bar{z}^{\alpha},(1-\alpha)(\bar{z}^{2} - \bar{z}^{1})\bigr).
\]
In both cases, one verifies that for any $\theta\in[0,1]$, the
linear combination $z^1+\theta z^2$ is a convex combination of $\bar{z}^{1}$
and $\bar{z}^{2}$, so that we may furthermore take $\Theta = 1$ in \eqref{eq:strict-convexity-det}.
But then
\begin{align*}
  \check{\cJ}_{det}(\bar{z}^{1})
    - \check{\cJ}_{det}(\bar{z}^{\alpha})
  &> \rD \check{\cJ}_{det}(\bar{z}^{\alpha})[\alpha(\bar{z}^{1} - \bar{z}^{2})],\\
  \check{\cJ}_{det}(\bar{z}^{2})
    - \check{\cJ}_{det}(\bar{z}^{\alpha})
  &> \rD \check{\cJ}_{det}(\bar{z}^{\alpha})[(1-\alpha)(\bar{z}^{2} - \bar{z}^{1})].
\end{align*}
Multiplying the first inequality by $1-\alpha$, the second inequality by $\alpha$,
then adding up the two inequalities, and finally using linearity
we conclude that
\[
  (1-\alpha)\check{\cJ}_{det}(\bar{z}^{1})
  + \alpha\check{\cJ}_{det}(\bar{z}^{2})
  - \check{\cJ}_{det}(\bar{z}^{\alpha})
  > 0,
\]
which establishes strict convexity of $\check{\cJ}_{det}$.
\end{proof}

Now, we will conclude that $\check{v}$ is continuously differentiable.
We do so by showing that $\check{v}$ is strictly convex and semiconcave with linear modulus,
where we recall that the latter means that there is a constant $C>0$ such that
\[
  \check{v}(\check{\gamma}_0+h) + \check{v}(\check{\gamma}_0-h) - 2\check{v}(\check{\gamma}_0) \leq C|h|^2
\]
for all $\check{\gamma}_0,h\in\R$ such that $\check{\gamma}_0-h \geq \check{\gamma}_{\min}$; see \cite[Definition 1.1.1]{cannarsa2004semiconcave}.

\begin{theorem}
The reparametrized value function $\check{v}$ is strictly convex and semiconcave with linear modulus,
and hence continuously differentiable on $(\check{\gamma}_{\min},\infty)$ with Lipschitz derivative.
\end{theorem}

\begin{proof}
Semiconcavity of $\check{v}$ follows from concavity of $v$.
Indeed, recall that $\check{v} = v\circ\mathrm{inv}$, where $\mathrm{inv}:[\check{\gamma}_{\min},\infty]\to [0,\gamma_{\max}]$ is given by
$\mathrm{inv}(\check{\gamma}_0):= 1/\check{\gamma}_0$.
Note that $\mathrm{inv}$ is continuously differentiable with Lipschitz derivative
on $[\check{\gamma}_{\min},\infty)$.
Moreover, $v$ is concave and hence in particular semiconcave with linear modulus.
But then \cite[Proposition 2.1.12 (ii)]{cannarsa2004semiconcave} implies that
$\check{v} = v\circ\mathrm{inv}$ is semiconcave with linear modulus.
Regarding strict convexity, let $\check{\gamma}_0,\check{\gamma}_0'\in[\check{\gamma}_{\max},\infty)$
with $\check{\gamma}_0\neq\check{\gamma}_0'$ and fix $\alpha\in(0,1)$.
Moreover, define $\check{\gamma}_0^\alpha := \alpha\check{\gamma}_0 + (1-\alpha)\check{\gamma}_0'$.
Denote by $h,h',h^\alpha$ continuous optimal controls for 
$\check{\gamma}_0,\check{\gamma}_0'$ and $\check{\gamma}_0^\alpha$,
respectively, which exist by Theorem~\ref{thm:existence}.
Then we conclude that
\begin{align*}
  \check{v}(\check{\gamma}_0^{\alpha})
  = \check{\cJ}_{det}\bigl(h^{\alpha},\check{\gamma}_0^{\alpha}\bigr)
  &\leq \check{\cJ}_{det}\bigl(\alpha h + (1-\alpha)h',\check{\gamma}_0^{\alpha}\bigr)
  \\
 & < \alpha\check{\cJ}_{det}\bigl(h,\check{\gamma}_0\bigr)
    + (1-\alpha)\check{\cJ}_{det}\bigl(h',\check{\gamma}_0'\bigr)
  = \alpha\check{v}(\check{\gamma}_0) + (1-\alpha)\check{v}(\check{\gamma}_0')
\end{align*}
by the strict convexity of $\check{\cJ}_{det}$.
Thus $\check{v}$ is strictly convex on $[\check{\gamma}_{\min},\infty)$,
which extends to $[\check{\gamma}_{\min},\infty]$ by the continuity of $\check{v}$.
Finally, the continuous differentiability of $\check{v}$ with Lipschitz continuous derivative
is a consequence of the semiconcavity with linear modulus and the convexity of $\check{v}$;
see \cite[Theorem 3.3.7]{cannarsa2004semiconcave}.
\end{proof}

The following two corollaries are immediate from the strict convexity of
$h\mapsto\check{\cJ}_{det}(h;\check{\gamma}_0)$ for $\check{\gamma}_0\in[\check{\gamma}_{\min},\infty)$
and the fact that \eqref{eq:reparameterized-control-problem} is a reparametrization of
\eqref{eq:def-reduced-value-function}.

\begin{corollary}\label{cor:unique-optimal-control}
For every initial condition $\gamma_0\in(0,\gamma_{\max}]$, there exists a unique continuous control 
$h^*\in\cA^{obs}_{det}$ for the reduced control problem \eqref{eq:def-reduced-value-function}
with value function $v(\gamma_0)$.
\end{corollary}

\begin{corollary}\label{cor:regularity-original-value-function}
The value function $v$ is continuously differentiable on $(0,\gamma_{\max})$.
\end{corollary}

\subsection{Characterization of the Optimal Control and Verification Theorems}\label{sec:verification}
With the differentiability of $v$ and the existence of a unique continuous control eastablished,
the next step is to identify the optimal control in terms of a feedback function.
We begin by introducing a transform similar to the Legendre--Fenchel transformation
of the cost function $c$.
Define $c^{*}:[0,M_{0}]\to[0,\infty)$ as
\begin{equation*}
	c^{*}(x):=\max_{h\in[0,h_{\max}]}\bigl\{hx-c(h)\bigr\}.
\end{equation*}
The following proposition gathers some useful properties of $c^*$.

\begin{proposition}\label{prop:convex-conjugate}
For $x\in[0,M_0]$ and with $\fh:[0,M_0]\to[0,h_{\max}]$ defined as
\[
	\fh(x) :=\arg\max_{h\in[0,h_{\max}]}\bigl\{hx-c(h)\bigr\}
		=\begin{cases}
			0, & \text{if }x< c'(0),\\
			(c')^{-1}(x), & \text{if }x\geq c'(0),
		\end{cases}
\]
it holds that
 \begin{align*}
		c^{*}(x) & =\begin{cases}
			-c(0), & \text{if }x< c'(0),\\
			(c')^{-1}(x) x - c\bigl((c')^{-1}(x)\bigr),
			& \text{if }x\geq c'(0),
		\end{cases}\\
		(c^{*})'(x)	& =\begin{cases}
			0, & \text{if }x< c'(0),\\
			(c')^{-1}(x), & \text{if }x\geq c'(0),
		\end{cases}\\
		\fh'(x)&= \begin{cases}
			0, & \text{if }x< c'(0),\\
			\frac{1}{c''((c')^{-1}(x))}, & \text{if }x\geq c'(0).
		\end{cases}
	\end{align*}
Furthermore, $-c^{*}$ is concave and even strictly concave on $[c'(0),M_{0}]$
and $\fh$ is increasing and even strictly increasing on $[c'(0),M_{0}]$.
\end{proposition}

\begin{proof}
First observe that, by definition of $h_{\max}$, it holds that
\[
    c'(h_{\max}) = c'\bigl((c')^{-1}(M_0)\bigr) = M_0.
\]
Since $c'$ is (strictly) increasing, it follows that $c'([0,h_{\max}]) = [c'(0),M_0]$. But as
\[
  0 = \frac{\rd}{\rd h}\bigl[hx-c(h)\bigr] = x - c'(h),\quad h\in[0,\infty),
\]
we conclude that the maximizer in $c^{*}$ can be computed using the first-order condition
if $x\in[c'(0),M_0]$ and is equal to zero otherwise.
This yields the desired representation of $\fh$.
The representation of $c^*$ is immediate by plugging in the optimizer $\fh$.
The representation of $(c^*)'$ follows from differentiation and the inverse function theorem.
Finally, the expression for $\fh'$ is again obtained using the inverse function theorem
and the properties of $-c^*$ and $\fh$ are a immediate from the monotonicity
and convexity assumptions imposed on $c$. 
\end{proof}

We are now ready for the first of two verification theorems.
To begin with, let us recall that the HJB equation for the reduced problem is given by
\[
	- \delta w + \inf_{h\in[0,\infty)}\bigl\{
	f(\argdot,h)w' + \bar{a}\gamma + c(h)\bigr\} = 0\quad\text{on }[0,\gamma_{\max}]
\]
We will shortly show that $v$ solves the HJB equation.
Once this has been done, using the optimizer $\fh$ of the Legendre--Fenchel transform
given by Proposition~\ref{prop:convex-conjugate}, it follows that the infimum is attained
by $H^*:[0,\gamma_{\max}]\to[0,h_{\max}]$ given by
\begin{equation}\label{eq:feedback-map}
  H^{*}(\gamma_0)
  := \fh\bigl(\gamma_0^{2}v'(\gamma_0)\bigr),
  \quad \gamma_0\in [0,\gamma_{\max}],
\end{equation}
where $v'(0):=v'(0+)$ and $v'(\gamma_{\max}):=v'(\gamma_{\max}-)$ are taken to be the
right and left derivatives at $0$ and $\gamma_{\max}$, respectively.
These one-sided derivatives exist since $v$ is concave.
The function $H^*$ is our candidate for the optimal feedback control.
We begin by constructing a control and a controlled state from $H^*$.

\begin{lemma}\label{lem:optimal-state}
For any $\gamma_0\in[0,\gamma_{\max}]$, the ordinary differential equation
\[
  \rd\gamma^{*;\gamma_0}_t = f\bigl(\gamma^{*;\gamma_0}_{t},H^*(\gamma^{*;\gamma_0}_{t})\bigr)\rd t,\quad t\in[0,\infty),\quad \gamma^{*;\gamma_0}_0 = \gamma_0,
\]
admits a unique $[0,\gamma_{\max}]$-valued solution.
Moreover, $h^{*;\gamma_0} := H^*(\gamma^{*;\gamma_0})\in\cA^{obs}_{det}$ is continuous
and $\gamma^{*;\gamma_0}=\gamma^{h^{*;\gamma_0};\gamma_0}$.
\end{lemma}

\begin{proof}
With the right-hand side of the differential equation being locally Lipschitz continuous in $\gamma$,
it follows that it admits at least a local solution.
Moreover, since $H^*$ is $[0,h_{\max}]$-valued, it follows as in the proof of 
Lemma~\ref{lem:gamma-bounded} that this local solution is non-negative and upper bounded
by the uncontrolled conditional variance $\gamma^0$, hence $[0,\gamma_{\max}]$-valued.
Thus, in particular, $\gamma^{*;\gamma_0}$ exists $[0,\infty)$ and is uniquely determined.
With this, it is obvious that $h^* := H^*(\gamma^{*;\gamma_0})\in\cA^{obs}_{det}$,
and continuity follows from the continuity of both $\fh$ and $v'$.
The identity $\gamma^{*;\gamma_0}=\gamma^{h^*;\gamma_0}$ follows directly from definition of $h^*$ and uniqueness.
\end{proof}

We now complete the discussion of the reduced problem by showing that the value function $v$
solves the associated HJB equation \eqref{eq:HJB-reduced-problem}
and that the feedback control $h^*$ constructed in Lemma~\ref{lem:optimal-state}
is the unique continuous optimal control.

\begin{theorem}
\label{th:verification-deterministic}
The value function $v$ solves the reduced HJB equation \eqref{eq:HJB-reduced-problem}.
Moreover, for any initial condition $\gamma_0\in[0,\gamma_{\max}]$,
the feedback control $h^{*;\gamma_0}$ constructed in Lemma~\ref{lem:optimal-state}
is the unique continuous optimal control.
\end{theorem}

\begin{proof}
Following \cite[Proposition III.2.5]{bardi1997optimal},
the value function satisfies the dynamic programming principle (DPP)
\[
  v(\gamma_0) = \inf_{h\in\cA^{obs}_{det}} \Bigl[\int_0^t e^{-\delta s} k\bigl(\gamma_s^{h;\gamma_0},h_s\bigr)\rd s + e^{-\delta t}v\bigl(\gamma_s^{h;\gamma_0}\bigr)\Bigr],\quad t\in[0,\infty).
\]
Now let $h$ be any continuous control.
The DPP implies
\[
  \int_{0}^{t}e^{-\delta s}k(\gamma_s^{h;\gamma_0},h_s)\rd s
  \geq -\bigl(e^{-\delta t}v(\gamma^{h;\gamma_0}_{t}) - v(\gamma_0)\bigr)\\
  = \int_0^te^{-\delta s}\bigl(\delta v(\gamma^{h;\gamma_0}_{s}) - v'(\gamma^{h;\gamma_0}_{s})f(\gamma^{h;\gamma_0}_{s},h_s) \bigr)\rd s
\]
or, equivalently,
\[
  0 \leq \frac{1}{t}\int_0^t e^{-\delta s}\bigl( - \delta v(\gamma^{h;\gamma_0}_{s}) + f(\gamma^{h;\gamma_0}_{s},h_s)v'(\gamma^{h;\gamma_0}_{s}) + k(\gamma_s^{h;\gamma_0},h_s)\bigr) \rd s,\quad t>0,
\]
with equality if $h=h^*$ is the optimal continuous control.
By the mean value theorem and upon sending $t\downarrow 0$, it follows that
\begin{equation}\label{eq:feedback}
 0 \leq - \delta v(\gamma_0) + f(\gamma_0,h_0)v'(\gamma_0) + k(\gamma_0,h_0)
\end{equation}
with equality if $h_0 = h_0^*$ is the initial value of the optimal control.
We have therefore argued that
\[
  0 = - \delta v(\gamma_0) + \inf_{h\in[0,h_{\max}]}\bigl\{f(\gamma_0,h)v'(\gamma_0) + k(\gamma_0,h)\bigr\},\quad \gamma_0\in[0,\gamma_{\max}],
\]
which is to say that $v$ indeed solves the HJB equation.
Finally, notice that the feedback control $h^{*;\gamma_0}$ is constructed precisely
such that we have equality in \eqref{eq:feedback},
from which we conclude by standard arguments that $h^{*;\gamma_0}$ is optimal.
\end{proof}

Having solved the reduced problem, we can now return to the full information problem.
First, we observe that $v$ does indeed satisfy the HJB equation 
\[
  0 = - \delta v(\gamma_0) + \inf_{h\in[0,\infty)}\bigl\{f(\gamma_0,h)v'(\gamma_0) + k(\gamma_0,h)\bigr\},\quad \gamma_0\in[0,\gamma_{\max}],
\]
where $h$ is not restricted to the compact set $[0,h_{\max}]$.
This follows from the fact that
\[
  h \mapsto \frac{\rd}{\rd h} \bigl[ f(\gamma_0,h)v'(\gamma_0) + k(\gamma_0,h) \bigr]
  = -\gamma_0^2v'(\gamma_0) + c'(h)
\]
is strictly increasing on $[0,\infty)$ with its unique root contained in $[0,h_{\max}]$.
Next, recall that, in the full information case, the HJB equation takes the form
\[
  -\delta W + \inf_{(u,h)\in\cU}\Bigl\{\cL^{u,h}W + \frac{1}{2}\bigl(\kappa x^{2} + \rho u^{2}\bigr) + c(h)\Bigr\} = 0\qquad\text{on }\S=\R^2\times[0,\gamma_{\max}],
\]
where $\cU=\R\times[0,\infty)$.
From the discussion in Section~\ref{sec:reduction} and the fact that $v$ solves the HJB equation of the reduced problem,
we conclude that a classical solution of the HJB equation in the full information case is given by the function $W:\S\to\R$ defined as
\begin{equation}\label{eq:hjb-solution}
  W(x,m,\gamma) = a_1 x^2 + a_2m^2 + a_3 xm + b_1 x + b_2m + a_2\gamma + \frac{1}{\delta}\bigl(C_1 + a_2\sigma_2^2\bigr) + v(\gamma),\quad (x,m,\gamma)\in\S;
\end{equation}
see \eqref{eq:ansatz} and \eqref{eq:ansatz-w}.
With this, the optimizers of the infimum in the full information HJB equation are given
by the feedback functions $U^*:\R^2\to\R$ and $H^*:[0,\infty)\to[0,\infty)$ defined as
\begin{equation}\label{eq:optimal-feedback-maps}
  U^*(x,m) := - \frac{2a_1 x + a_3 m + b_1}{\rho}
  \quad\text{and}\quad
  H^*(\gamma) := \fh\bigl(\gamma^2v'(\gamma)\bigr),
  \qquad (x,m,\gamma)\in\S,
\end{equation}
where the function $\hat h$ is defined in Proposition \ref{prop:convex-conjugate}.
Observe that $H^*$ is the feedback map in \eqref{eq:feedback}.

\begin{lemma}\label{lem:optimal-state-full}
For any $z=(x,m,\gamma)\in\S$, the stochastic differential equation
\[
  \rd Z^{*;z}
  = \begin{pmatrix}
      m^{*;z}_{t} + U^*(X^{*;z}_t,m^{*;z}_t)\\
      \lambda(\bar{\mu} - m^{*;z}_{t})\\
      f\bigl(\gamma^{*;z}_{t},H^*(\gamma^{*;z}_t)\bigr)
	\end{pmatrix}\rd t
  + \begin{pmatrix}
      \sigma_{1} & 0\\
      \bar{\sigma}_{1}\gamma^{*;z}_{t} & \sqrt{H^*(\gamma^{*;z}_t)}\gamma^{*;z}_t\\
      0 & 0
    \end{pmatrix}\rd I^{*;z}_{t},
  \quad Z^{*;z}_0 = z,
  \quad t\in[0,\infty),
\]
with
\[
  \rd I^{*;z}_t
  = - \begin{pmatrix}
      \bar{\sigma}_{1}(m_{t}^{*;z} - \mu_{t})\\
      (m_{t}^{*;z} - \mu_{t})\sqrt{H^*(\gamma^{*;z}_t)}
    \end{pmatrix} \rd t
  + \begin{pmatrix}
      \rd B^{1}_{t}\\
      \rd B^{3}_{t}
    \end{pmatrix},
  \quad I^{*;z}_0 = 0,
  \quad t\in[0,\infty),
\]
admits a unique square-integrable solution $Z^{*;z} = (X^{*;z}_t,m^{*;z}_t,\gamma^{*;z}_t)$
taking values in $\S$.
Moreover, setting $u^{*;z} := U^*(X^{*;z},m^{*;z})$ and $h^{*;z}:= H^*(\gamma^{*;z})$
defines a pair of admissible controls $(u^*,h^*)\in\cA$.
Finally, $(u^*,h^*)$ is progressively measurable with respect to the filtration
$\cY^{0,h^*}\subseteq \cY^{u^*,h^*}$, implying that $\cY^{0,h^*}=\cY^{u^*,h^*}$.
\end{lemma}

\begin{proof}
Throughout this proof, we suppress the initial condition $z$ to simplify the notation.
The existence of the third component of the solution $\gamma^{*}$ has already been established
in Lemma~\ref{lem:optimal-state}.
Using $h^{*}= H^*(\gamma^{*})$ and the definition of $I^{*}$,
the equations for the first two components become
\[
  \rd \begin{pmatrix}X^{*}_t\\ m^{*}_t\end{pmatrix}
  = \begin{pmatrix}
      \mu_{t} + U^*(X^{*},m^{*})\\
      \lambda(\bar{\mu} - m^{*}_{t}) - \gamma^{*}_{t}(m_{t}^{*} - \mu_{t})(\bar{\sigma}_{1}^2 + h^{*}_t)
	\end{pmatrix}\rd t
  + \begin{pmatrix}
      \sigma_{1} & 0\\
      \bar{\sigma}_{1}\gamma^{*}_{t} & \sqrt{h^{*}_t}\gamma^{*}_t
    \end{pmatrix}
    \begin{pmatrix}
      \rd B^{1}_{t}\\
      \rd B^{3}_{t}
    \end{pmatrix}.
\]
With $U^*(X^{*},m^{*})$ being linear in $(X^{*},m^{*})$,
we conclude that a unique solution in the filtration $\F$ generated by $\mu_0$
and the Brownian motion $B$ exists.
However, rewriting the dynamics of $Z^*$ again in terms of the innovations process $I^*$,
it follows that $Z^*$ and $(u^*,h^*)$ are also adapted to the filtration generated by $I^*$,
which is a sub-filtration of $\cY^{u^*,h^*}$ as $I^*$ is $\cY^{u^*,h^*}$-adapted.
But $I^*$ does not depend on the control $u^*$, and hence $(Z^*,u^*,h^*)$
is also adapted to $\cY^{0,h^*}$.
As $\cY^{u^*,h^*}$ is the smallest filtration making $Z^*$ adapted we conclude that $\cY^{0,h^*}=\cY^{u^*,h^*}$.

To show that $(u^*,h^*)$ is admissible, it remains to verify that $e^{-\delta t}\E[(X_t^*)^2]\to 0$
as $t\to\infty$.
For this, we fix $t\in[0,\infty)$ and observe that $m^{*}_t$ is given explicitly by
\[
  m_t^{*} = e^{-\lambda t}\Bigl[m + \bar{\mu}\bigl(e^{\lambda t}-1\bigr)
    + \bar{\sigma}_1\int_0^t e^{\lambda s}\gamma^{*}_s\rd I^{1,*}_s
    + \int_0^t e^{\lambda s}\sqrt{h^{*}_s}\gamma^{*}_s\rd I^{2,*}_s \Bigr].
\]
Using this, a direct computation shows that
\[
  \E\bigl[(m_t^*)^2\bigr]
  = e^{-2\lambda t}\Bigl[\Bigl(m + \bar{\mu}\bigl(e^{\lambda t}-1\bigr)\Bigr)^2
    + \bar{\sigma}_1^2\int_0^t e^{2\lambda s}\bigl(\gamma^{*}_s\bigr)^2\rd s
    + \int_0^t e^{2\lambda s}h^{*}_s\bigl(\gamma^{*}_s\bigr)^2\rd s
    \Bigr].
\]
In particular, since both $h^*$ and $\gamma^*$ are bounded, there exists a constant $C>0$ such that
\[
  \sup_{t\in[0,\infty)}\E\bigl[(m_t^*)^2\bigr] \leq C.
\]
Similarly, $X^*_t$ is given explicitly by
\[
  X^*_t = e^{-\frac{2a_1}{\rho}t}\Bigl[x - \frac{1}{\rho}\int_0^t \bigl((a_3-\rho)m^*_s + b_1\bigr)e^{\frac{2a_1}{\rho}s}\rd s + \sigma_1\int_0^t e^{\frac{2a_1}{\rho}s} \rd I^{1,*}_s\Bigr],
\]
so that the Cauchy-Schwarz inequality yields 
\begin{align*}
  \E\bigl[(X_t^*)^2\bigr]
  &\leq 3 e^{-\frac{4a_1}{\rho}t}\Bigl[
    x^2 + \frac{2t}{\rho^2}\int_0^t \bigl((a_3-\rho)^2\E\bigl[(m^*_s)^2\bigr] + b_1^2\bigr)e^{\frac{4a_1}{\rho}s}\rd s
    + \sigma_1^2\int_0^t e^{\frac{4a_1}{\rho}s} \rd s
  \Bigr].
\end{align*}
Since the second moments of $m^*$ are uniformly bounded,
we conclude that there exists a constant $D>0$ such that
\[
  \E\bigl[(X_t^*)^2\bigr] \leq D(1+t),\quad t\in[0,\infty).
\]
In particular, it follows that
\[
  \lim_{t\to\infty}e^{-\delta t}\E\bigl[(X_t^*)^2\bigr] = 0,
\]
which concludes the proof.
\end{proof}

Having constructed a classical solution of the HJB equation and a candidate optimal feedback control,
the solution of the full information problem follows by standard arguments.

\begin{theorem}
The solution $W$ of the full information HJB equation \eqref{eq:HJB} given by \eqref{eq:hjb-solution} 
coincides with the value function $V$ of the full information problem given by
\eqref{eq:FI-value-function}.
Moreover, for any initial condition $z\in\S$,
the feedback control $(u^{*;z},h^{*;z})$ constructed in Lemma~\ref{lem:optimal-state-full} is optimal.
As a consequence, $(u^{*;z},h^{*;z})$ is also optimal for the original partial information problem
\eqref{eq:original-control-problem}.
\end{theorem}

\begin{proof}
This is a standard verification result and follows along classical arguments.
More precisely, according to \cite[Theorem III.9.1]{fleming2006controlled},
it suffices to show that
\begin{equation}\label{eq:liminf-proof-verification}
  \liminf_{t\to\infty} e^{-\delta t}\E\bigl[W(Z^{u,h}_t)\bigr] \leq 0
	\quad\text{for all }(u,h)\in\cA
\end{equation}
and
\begin{equation}\label{eq:limsup-proof-verification}
  \limsup_{t\to\infty} e^{-\delta t}\E\bigl[W(Z^{*}_t)\bigr] \geq 0,
\end{equation}
where $Z^{*}$ denotes the candidate optimal state process under the control $(u^{*;z},h^{*;z})$
and where we have suppressed the dependence of the state processes $Z^{u,h}$ and $Z^*$
on the initial condition $z$ in our notation.
First, observe that \eqref{eq:limsup-proof-verification} is trivial as $W\geq 0$.
Hence, we only have to verify \eqref{eq:liminf-proof-verification}.
For this, fix $(u,h)\in\cA$.
Note that $\gamma^{h}$ is bounded by Lemma~\ref{lem:gamma-bounded}, and
\[
  \lim_{t\to\infty} e^{-\delta t}\E\bigl[(X^{u}_t)^2\bigr] = 0
\]
by the admissibility of $(u,h)$; see \eqref{eq:pre-admissible}. Since $W$ is quadratic in $(x,m)$, it therefore suffices to show that
\[
  \lim_{t\to\infty} e^{-\delta t}\E\bigl[(m^{u,h}_t)^2\bigr] = 0.
\]
Just as in the proof of Lemma~\ref{lem:optimal-state-full}, one finds that
\[
  \E\bigl[(m_t^{u,h})^2\bigr]
  = e^{-2\lambda t}\biggl[\Bigl(m + \bar{\mu}\bigl(e^{\lambda t}-1\bigr)\Bigr)^2
    + \bar{\sigma}_1^2\E\Bigl[\int_0^t e^{2\lambda s}\bigl(\gamma^{h}_s\bigr)^2\rd s\Bigr]
    + \E\Bigl[\int_0^t e^{2\lambda s}h_s\bigl(\gamma^{h}_s\bigr)^2\rd s\Bigr]
    \biggr].
\]
Now $\gamma^{h}$ is bounded, so we conclude that there exists a constant $C>0$ such that
\[
  e^{-\delta t}\E\bigl[(m_t^{u,h})^2\bigr] \leq e^{-\delta t}C\Bigl(1 + \E\Bigl[\int_0^t h_s \rd s\Bigr]\Bigr).
\]
But the right-hand side vanishes as $t\to\infty$ by the admissibility of $h$, hence concluding the proof.
\end{proof}

\section{Properties of the Optimal Information Acquisition Rate}\label{sec:properties}
We now turn to the study of the qualitative properties of the reduced and full information problem.
Specifically, we are interested in the properties of the feedback map $H^{*}$
and the value function $v$ of the reduced problem.
In addition, we analyze the behavior of the dynamical system consisting of the optimal control $h^{*}$
and the optimally controlled conditional variance $\gamma^{*}$.
In particular, we show that $\gamma^*$ converges to a unique equilibrium $\equi$
and study the behavior of $h^*$ and $\gamma^*$ near $\equi$.
We begin in Section~\ref{sec:properties-reduced} by restricting our considerations
to the reduced problem \eqref{eq:def-reduced-value-function}.
In Section~\ref{sec:properties-original}, we then study the implications
for the full information problem \eqref{eq:FI-value-function}.

\subsection{Qualitative Properties of the Reduced Problem}\label{sec:properties-reduced}
We begin our study of the optimal information acquisition rate by taking a closer look
at the properties feedback map $H^*$.
\begin{lemma}\label{lem:feedback-map}
Defining $\gamma_{\cD}\in[0,\gamma_{\max}]$ as
\begin{equation}
    \gamma_{\cD} := \begin{cases}
        1/(\check{v}')^{-1}(-c'(0))\quad \textup{ if } -c'(0)\in\check{v}'([0,\gamma_{\max}]),\\
        \gamma_{\max} \quad \textup{ else},
    \end{cases}
\end{equation}
we have that $H^{*} = 0$ on $[0,\gamma_{\cD}]$ and $H^*$ is strictly increasing on $[\gamma_{\cD},\gamma_{\max}]$.
\end{lemma}

\begin{proof}
From the relation $v(\gamma) = \check{v}(1/\gamma)$ we infer that $\gamma^2 v'(\gamma) = -\check{v}'(1/\gamma)$.
With this, it follows that $\gamma_{\cD}$ is the unique real number in
$[0,\gamma_{\max}]$ satisfying
\[
  \gamma_{\cD}^2 v'(\gamma_{\cD}) = -\check{v}'(1/\gamma_{\cD}) = c'(0).
\]
Moreover, since $\check{v}'$ is strictly increasing, we conclude that also
$\gamma\mapsto \gamma^2 v'(\gamma)$ is strictly increasing.
Finally, it was shown in Proposition~\ref{prop:convex-conjugate} that the mapping
$\hat{h}$ in the definition of $H^*$ is zero on $[0,c'(0)]$ and strictly increasing on $[c'(0),M_0]$.
Since $H^*(\gamma) = \hat{h}(\gamma^2v'(\gamma))$ the result follows.
\end{proof}

\begin{remark}
\label{rmk:power_threshhold}
    Note that $\gamma_{\cD} = 0$ in Lemma \ref{lem:feedback-map} for a power cost
    $c(x) = \zeta x^{1+\epsilon}$ with $\epsilon >0$.
    Indeed, since $c'(0) = 0$ and $\check{v}$ is strictly convex,
    we see that $-\check{v}'<0 = c'(0)$ on $(\check{\gamma}_{\min},\infty)$.
    Moreover, $\gamma^2 v'(\gamma) = 0$ if $\gamma = 0$,
    so we conclude that $\gamma_{\cD} = 0$ in that case.
\end{remark}

\begin{lemma}
\label{lem:equilibrium-functional}
There exists a unique number $\equi\in[0,\gamma_{\infty}^0]$ with $f(\equi,H^{*}(\equi))=0$,
where $\gamma_\infty^0$ is defined in \eqref{eq:gamma-infty}.
Furthermore, $f(\gamma,H^{*}(\gamma))<0$ for $\gamma\in(\equi,\gamma_{\max}]$
and $f(\gamma,H^{*}(\gamma))>0$ for $\gamma\in[0,\equi)$.
\end{lemma}

\begin{proof}
	As the optimal feedback map $H^{*}$ is nonnegative and increasing in
    $[0,\gamma_{\max}]$	by Lemma \ref{lem:feedback-map},
    we deduce from \eqref{eq:def-f} that the state coefficient function along the optimal control
	$\gamma\mapsto f(\gamma,H^{*}(\gamma))$ is strictly decreasing and hence has at most one root.
    As $f(0,H^{*}(0))\geq0$ and $f(\gamma_\infty,H^{*}(\gamma_\infty)) \leq f(\gamma_\infty,0) = 0$,
    continuity and the intermediate value theorem imply existence of a root in $[0,\gamma_\infty]$.
\end{proof}
Having established these properties of the optimal feedback function $H^*$ and the coefficient function
$f(\argdot,H^*(\argdot))$ of the optimal state,
we now study the implications for the optimal control and the optimal conditional variance.

\begin{proposition}
\label{prop:properties-equilibrium-state-control}
Let $\gamma_0\in[0,\gamma_{\max}]$ and denote by $h^*$ and $\gamma^*$ the associated optimal control
and optimal conditional variance, respectively.
Moreover, recall the constant $\gamma_{\cD}$ defined in Lemma~\ref{lem:feedback-map}.
\begin{enumerate}[label=\roman*),ref=\roman*)]
\item\label{item:proposition-properties-equilibrium-state-control-1}
  It holds that $\lim_{t\to\infty}\gamma^{*}_t = \equi$ and $\lim_{t\to\infty} h^*_t = \hEqui$
  for some $\equi\in [0,\gamma_{\max}]$ and $\hEqui\in[0,h_{\max}]$ independent of $\gamma_0$.
\item\label{item:proposition-properties-equilibrium-state-control-2}
  It holds that $\equi\to 0$ as $\sigma_{2}\to 0$ and $\equi = 0$ if and only if $\sigma_2 = 0$.
\item\label{item:proposition-properties-equilibrium-state-control-3}
  Both $\gamma^{*}$ and $h^{*}$ are monotone.
  Moreover, if $c'(0)\in[0,\gamma_{\max}]$, it holds that $h^{*}_{t}=0$ if $\gamma^{*}_{t}\in[0,\gamma_{\cD}]$ and $h^{*}_{t}>0$ if $\gamma^{*}_{t}\in(\gamma_{\cD},\gamma_{\max}]$.
\end{enumerate}
\end{proposition}

\begin{proof}
The function $\ell(\gamma):=(\gamma-\equi)^{2}$ satisfies $\ell(\gamma)>0$ for $\gamma\neq\equi$
and $\ell'(\gamma)f(\gamma,H^*(\gamma)) \leq 0$ for all $\gamma\in[0,\gamma_{\max}]$.
Hence $\ell$ is a Lyapunov function, and the convergence of $\gamma^*$ to $\equi$
follows from Lemma~\ref{lem:equilibrium-functional} and the Lyapunov theorem;
see e.g.\ \cite[Satz 5.6]{feichtinger2011optimale}.
Since $h^*$ is a function of $\gamma^*$, this also implies asymptotic stability of $h^*$,
i.e.\ its convergence to some equilibrium value $\hEqui\in[0,h_{\max}]$
independently of the initial condition $\gamma_0$.

Next, we recall that $\gamma^*\leq\gamma^0$, where $\gamma^0$ is the uncontrolled conditional variance
with corresponding equilibrium $\gamma_\infty^0$ defined in \eqref{eq:gamma-infty}.
Thus $\equi\leq\gamma_\infty^0$ and the statements $\equi\to 0$ as $\sigma_{2}\to 0$ and $\equi = 0$
if and only if $\sigma_2 = 0$ follow from the fact that the same is already true
if we replace $\equi$ by $\gamma_\infty^0$.

Finally, monotonicity of the state $\gamma^*$ is immediate from the fact that
$f(\gamma,H^{*}(\gamma))>0$ for $\gamma<\equi$
and $f(\gamma,H^{*}(\gamma))<0$ for $\gamma>\equi$.
Monotonicity of the control $h^*$ is immediate from monotonicity of the corresponding
controlled state and monotonicity of the feedback map $H^*$ as shown in Lemma~\ref{lem:feedback-map}.
The claim for the control being zero in $[0,\gamma_{\cD}]$ and strictly positive in the complement
is another direct consequence of Lemma~\ref{lem:feedback-map}.
\end{proof}

Proposition~\ref{prop:properties-equilibrium-state-control} shows that,
independently of the initial condition, the optimally controlled state converges
to the unique equilibrium $\equi$.
The same is true for the optimal control with the equilibrium value $\hEqui$.
Moreover, the uncertainty of the system, as measured by $\gamma$, converges to zero
as $t\to\infty$ if and only if there is no noise present in the hidden state $\mu$,
that is, if $\sigma_2=0$.
Next, regarding the convergence $\equi\to 0$ as $\sigma_2\to 0$, we shall see
in Section~\ref{subsec:sensitivity} that the convergence is even monotone.
Finally, Proposition~\ref{prop:properties-equilibrium-state-control} furthermore
shows that if $c'(0)\in[0,\gamma_{\max}]$, there is a critical level $\gamma_{\cD}$
such that it is optimal not to conduct any additional experiments ($h^*=0$) if
the conditional variance $\gamma^*$ is below this threshold.

Our next aim is to study the properties of the optimal control and of the value
function near the equilibrium.

\begin{proposition}\label{prop:value-equilibrium}
At the equilibrium $\equi$, the value function is given by
\begin{equation}\label{eq:value-function-at-equilibrium}
  \vEqui :=v(\equi)
  = \frac{1}{\delta}\bigl(\bar{a}\equi+ c(\hEqui)\bigr).
\end{equation}
Furthermore, if $\equi\in[0,\gamma_{\cD}]$, we have $\hEqui=0$ and,
if we even have $\equi\in(0,\gamma_{\cD})$,
\begin{equation}\label{eq:value-function-derivative-at-equilibrium-in-cD}
  v'(\equi)
  = \frac{\bar{a}}{2\bar{\sigma}_{1}^{2}\equi+2\lambda+\delta}.
\end{equation}
Conversely, if $\equi\in(\gamma_{\cD},\gamma_{\max}]$, we have
\begin{align}
  v'(\equi) &= \frac{1}{\equi^{2}} c'\left(
    -\frac{1}{\equi^{2}} (2\lambda\equi - \sigma_{2}^{2})
    -\bar{\sigma}_{1}^{2}
  \right),\label{eq:value-function-derivative-at-equilibrium-in-cDc}\\
  \hEqui &= -\frac{1}{\equi^{2}} (2\lambda\equi - \sigma_{2}^{2})
    - \bar{\sigma}_{1}^{2}.\label{eq:control-at-equilibrium-in-cDc}
\end{align}
\end{proposition}

\begin{proof}
At the equilibrium, the optimal control $h^*=\hEqui$ and the corresponding optimal
state $\gamma^* = \equi$ are constant, from which we immediately find
\[
    \vEqui=v(\equi) = \cJ_{det}(\hEqui;\equi)
    = \frac{1}{\delta}\bigl(\bar{a}\equi+ c(\hEqui)\bigr).
\]
Now assume that $\equi\in(0,\gamma_{\cD})$, so that in particular $\gamma_{\cD}>0$.
We first observe that Lemma~\ref{lem:feedback-map} and
Proposition~\ref{prop:properties-equilibrium-state-control} imply that,
for any $\gamma_0\in[0,\gamma_{\cD}]$, the optimal control is $h^* = 0 = \hEqui$
and hence $v(\gamma_0) = \cJ_{det}(0;\gamma_0)$ for all $\gamma_0\in[0,\gamma_{\cD}]$.
Now $\gamma_0 \mapsto \cJ_{det}(0;\gamma_0)$ is twice continuously differentiable
(as seen in the proof of Theorem~\ref{thm:regularity-value-function}), and hence $v$
is twice continuously differentiable on $(0,\gamma_{\cD})$.
Next, with $h^*=0$, the HJB equation on $(0,\gamma_{\cD})$ simplifies to
\[
  0 = \bigl(-\bar{\sigma}_1^2\gamma_0^2 - 2\lambda\gamma_0 + \sigma_2^2\bigr)v'(\gamma_0)
    + \bar{a}\gamma_0 + c(0) - \delta v(\gamma_0),
  \quad \gamma_0\in(0,\gamma_{\cD}).
\]
Differentiating with respect to $\gamma_0$ and setting $\gamma_0 = \equi$, it
follows that
\[
  0 = \bigl(-\bar{\sigma}_1^2\equi^2 - 2\lambda\equi + \sigma_2^2\bigr)v''(\equi)
    - \bigl(2\bar{\sigma}_1^2\equi + 2\lambda + \delta\bigr)v'(\equi)
    + \bar{a}.
\]
Now the coefficient preceding $v''(\equi)$ is equal to $f(\equi,0) = 0$, so
that solving for $v'(\equi)$ yields
\eqref{eq:value-function-derivative-at-equilibrium-in-cD} as claimed.
For the case $\equi\in[\gamma_{\cD},\gamma_{\max}]$, the identity
\eqref{eq:control-at-equilibrium-in-cDc} follows directly from
\[
  0 = f(\equi,\hEqui) = -(\bar{\sigma}_1^2+\hEqui)\equi^2 - 2\lambda\equi + \sigma_2^2
\]
by solving for $\hEqui$.
Finally, observe that
\[
  \hEqui = H^*(\equi) = \hat{h}\bigl(\equi^2v'(\equi)\bigr)
  = (c')^{-1}\bigl(\equi^2v'(\equi)\bigr),
\] 
where the last identity is a consequence of $\equi\in[\gamma_{\cD},\gamma_{\max}]$.
Solving for $v'(\equi)$ and using \eqref{eq:control-at-equilibrium-in-cDc}
yields \eqref{eq:value-function-derivative-at-equilibrium-in-cDc} as claimed.
\end{proof}

We conclude the discussion in the reduced problem by studying the asymptotic
behaviour as $\gamma\to 0$ and as $\gamma\to\infty$.
For this, let us first observe that the solution
of the reduced problem does not depend on the particular choice of $\gamma_{\max}$,
provided that it is chosen sufficiently large.
More precisely, if
$\gamma_{\max} < \hat{\gamma}_{\max}$ are two different choices
for the upper bound on the conditional variance and if $v$ and $\hat{v}$ denote the
corresponding value functions, then $v = \hat{v}$ on $[0,\gamma_{\max}]$ and the
optimal controls for any initial condition $\gamma_0\in[0,\gamma_{\max}]$
coincide.
In light of this, with a slight abuse of notation, we can subsequently
assume that the value function $v$ of the reduced problem is defined on all of $[0,\infty)$.
Note that $v$ is still Lipschitz continuous on $[0,\infty)$ with constant $L_v = \bar{a}/\delta$.

\begin{theorem}
\label{eq:quadratic_asymptotic}
The feedback map $H^*$ satisfies
\[
  0 \leq H^*(\gamma) \leq \hat{h}\bigl(\gamma^2 L_v\bigr)
  \leq (c')^{-1}\bigl(\gamma^2 L_v\bigr),
  \quad\gamma\in[0,\infty).
\]
If, moreover, $c$ is a quadratic cost function of the form $c(x) = \zeta x^2$
for some $\zeta>0$, then
\[
  \lim_{\gamma\to\infty} \frac{H^*(\gamma)}{\gamma^{1/2}}
  = \frac{\sqrt{\bar{a}}}{\sqrt{\zeta}}.
\]
\end{theorem}

\begin{proof}
The first claim is immediate from the definition of the feedback map $H^*$,
the monotonicity of $\hat{h}$, the Lipschitz continuity of $v$, and the inequality
$\hat{h} \leq (c')^{-1}$.
Regarding the second claim, let us suppose that $c$ is quadratic
and recall that in this case $\gamma_{\cD} = 0$, meaning that
\[
  H^*(\gamma) = \hat{h}\bigl(\gamma^2 v'(\gamma)\bigr)
  = (c')^{-1}\bigl(\gamma^2 v'(\gamma)\bigr)
  = \frac{\gamma^2 v'(\gamma)}{2\zeta},\quad \gamma\in[0,\infty).
\]
Plugging this optimizer into the HJB equation \eqref{eq:HJB-reduced-problem} yields
\[
  0 = - \frac{1}{4\zeta}\gamma^4 v'(\gamma)^2
    + \Bigl(-\bar{\sigma}_1^2-\frac{2\lambda}{\gamma}+\frac{\sigma_2^2}{\gamma^2}\Bigr)\gamma^2v'(\gamma)
    + \bar{a}\gamma - \delta v(\gamma),\quad \gamma\in[0,\infty).
\]
Treating this equation as a quadratic equation in $\gamma^2v'(\gamma)$, the
quadratic formula shows that
\begin{equation}\label{eq:quadratic}
  \gamma^2v'(\gamma) = 2\zeta\Bigl(-\bar{\sigma}_1^2-\frac{2\lambda}{\gamma}+\frac{\sigma_2^2}{\gamma^2}\Bigr)
    + 2\sqrt{\zeta^2\Bigl(-\bar{\sigma}_1^2-\frac{2\lambda}{\gamma}+\frac{\sigma_2^2}{\gamma^2}\Bigr)^2 + \zeta(\bar{a}\gamma - \delta v(\gamma))},
\end{equation}
where we have used $v'(\gamma)>0$ to identify the correct root.
This along with $v\geq 0$ implies
\begin{equation*}
  \limsup_{\gamma\to \infty}\gamma^{3/2}v'(\gamma) \leq 2\sqrt{\zeta\bar{a}}.
\end{equation*}
In particular, for every $\eps>0$ there exists $\hat{\gamma}>0$ such that
\[
  v'(\gamma) \leq \bigl(\eps + 2\sqrt{\zeta\bar{a}}\bigr)\gamma^{-3/2},
  \quad \gamma\in[\hat{\gamma},\infty).
\]
But then we conclude that
\[
  v(\gamma)
  = v(\hat{\gamma}) + \int_{\hat{\gamma}}^\gamma v'(s)\rd s
  \leq v(\hat{\gamma}) + \bigl(\eps + 2\sqrt{\zeta\bar{a}}\bigr)\int_{\hat{\gamma}}^\gamma s^{-3/2}\rd s
  \leq v(\hat{\gamma}) + 2\bigl(\eps + 2\sqrt{\zeta\bar{a}}\bigr)\hat{\gamma}^{-1/2}.
\]
This shows that $v$ is in fact bounded, so that dividing both sides of \eqref{eq:quadratic}
by $\gamma^{1/2}$ we find
\[
  \lim_{\gamma\to\infty} \gamma^{3/2}v'(\gamma) = 2\sqrt{\zeta\bar{a}}.
\]
But then
\[
  \lim_{\gamma\to\infty}\frac{H^*(\gamma)}{\gamma^{1/2}}
  = \lim_{\gamma\to\infty}\frac{(c')^{-1}\bigl(\gamma^2v'(\gamma)\bigr)}{\gamma^{1/2}}
  = \lim_{\gamma\to\infty}\frac{\gamma^{3/2}v'(\gamma)}{2\zeta}
  = \frac{\sqrt{\bar{a}}}{\sqrt{\zeta}},
\]
which concludes the proof.
\end{proof}

\begin{figure}[ht]
\centering
\begin{subfigure}{.49\textwidth}
    \centering
    \includegraphics[width=.8\linewidth]{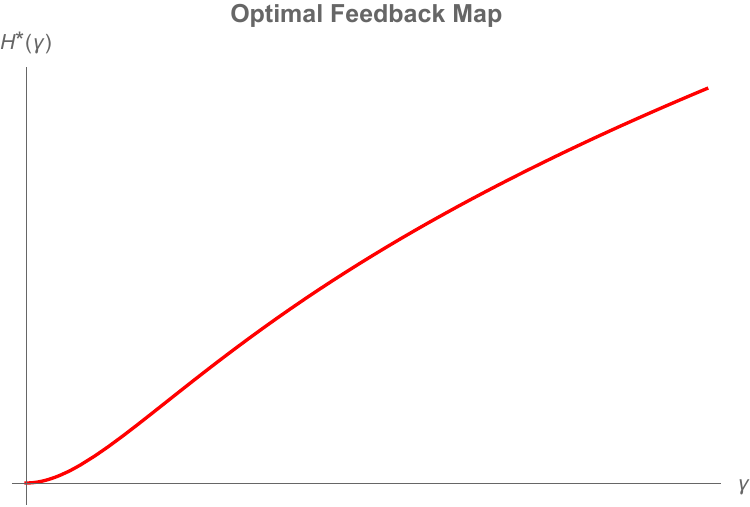}
    \caption{The optimal feedback map}
    \label{fig:feedback-map}
\end{subfigure}
\begin{subfigure}{.49\textwidth}
    \centering
    \includegraphics[width=.8\linewidth]{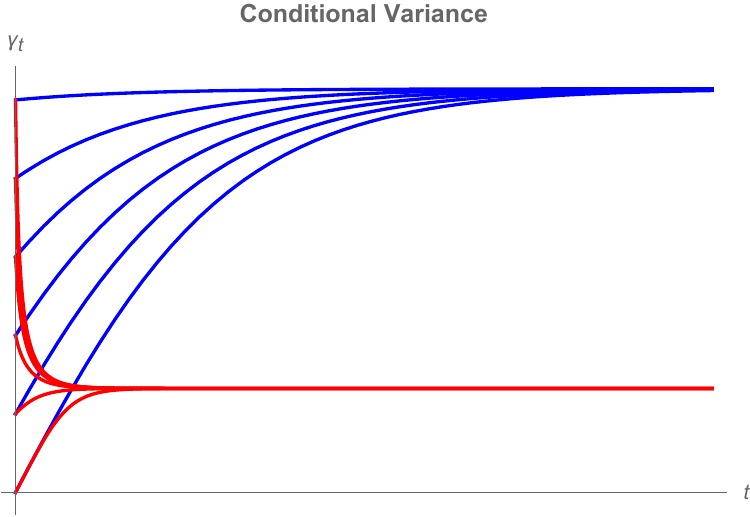}
    \caption{Optimal conditional (red) variance and uncontrolled conditional
    variance (blue) for different initial conditions}
    \label{fig:state}
\end{subfigure}
\par\bigskip
\begin{subfigure}{.49\textwidth}
    \centering
    \includegraphics[width=.85\linewidth]{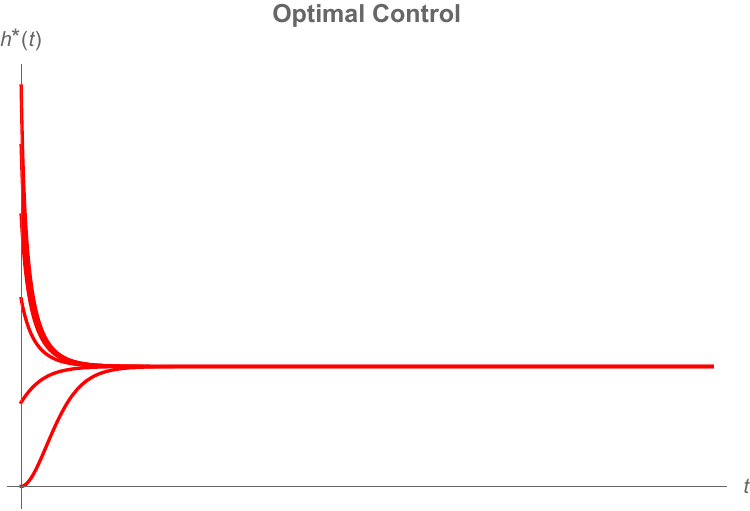}
    \caption{Realizations of the optimal control for different initial conditions}
    \label{fig:realizations-control}
\end{subfigure}
\begin{subfigure}{.49\textwidth}
    \centering
    \includegraphics[width=.85\linewidth]{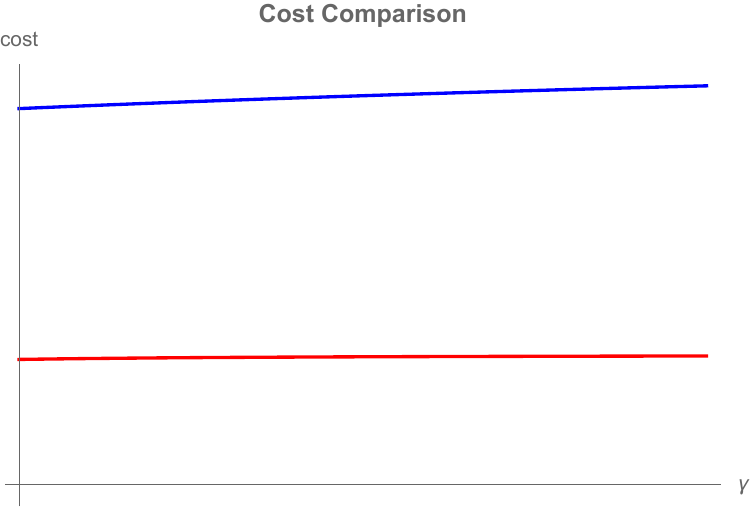}
    \caption{Comparison of the cost functional for optimal (red) and
	no-information-acquisition control (blue)}
    \label{fig:value-function}
\end{subfigure}
\caption{Illustrations for the reduced control problem.}
\label{fig:2x2-grid}
\end{figure}

We close this section with numerical illustrations in the case of a quadratic cost function.
Here, Figure~\ref{fig:feedback-map} depicts the optimal feedback map $H^*$,
whereas Figure~\ref{fig:state} shows the optimal conditional variance $\gamma^*$
(in red) and the uncontrolled conditional variance $\gamma^0$ (in blue) as functions
of time.
The latter illustration highlights in particular the convergence of $\gamma^*$ and $\gamma^0$
to the distinct equilibrium states $\equi$ and $\gamma_\infty^0$, respectively.
Next, Figure~\ref{fig:realizations-control} depicts the optimal control
$h^*$ as a function of time for different initial conditions $\gamma_0$, showing
that the convergence to the equilibrium state $\hEqui$ does in fact not depend on $\gamma_0$.
Finally, Figure~\ref{fig:value-function} compares the cost of the no-information-acquisition
control $h\equiv 0$ (in blue) and the cost of the optimal control $h^*$, that is,
the value function $v$, (in red).

\subsection{Impact of Information Acquisition}
\label{sec:properties-original}

In order to study the effect of information acquisition,
we proceed by comparing the original partial observation problem as given by
\eqref{eq:original-control-problem} with two benchmark problems in which the
\DM\ (i) can observe the process $\mu$ and (ii) does not observe $\mu$ and cannot
acquire additional information through experiments.

The observable case (i) is formally obtained by changing the set of
admissible controls from $\cA$ to $\cA^{pre}$.
As full observability implies that
there is no benefit in choosing a costly information acquisition rate $h$, we
may furthermore assume that $h\equiv 0$.
The value function of this problem may thus be written as
\[
  V^{full}(x,\mu_0) := \inf_{(u,0)\in\cA^{pre}}
    \E\Bigl[\int_0^\infty e^{-\delta t} \Bigl(\frac{1}{2}\kappa|X_t^u|^2
	  + \rho|u_t|^2 + c(0)\Bigr) \rd t\Bigr]
  \quad\text{subject to \eqref{eq:po-state}}
\]
for all $(x,\mu_0)\in\R^2$.
This problem is a standard linear quadratic stochastic control problem which can be solved explicitly.
Indeed, it is not difficult to show that
\[
  V^{full}(x,\mu_0) = a_{1}x^{2} + a_{2}\mu_0^{2} + a_{3}x\mu_0
    + b_{1}x + b_{2}\mu_0 + c^{full},
\]
where $a_1,a_2,a_3,b_1,b_2\in\R$ are the same coefficients as appearing in the value
function $V$; see \eqref{eq:ansatz-coefficients}.
Moreover, the constant $c^{full}$ is given explicitly by
\[
  c^{full} = \frac{1}{\delta}\Bigl(
    \sigma_{1}^{2} a_{1} + \sigma_{2}^{2} a_{2} + b_{2} \lambda \bar{\mu}
    - \frac{b_{1}^{2}}{2\rho} + c(0)
  \Bigr).
\]
Finally, the optimal feedback map for this problem coincides with
the feedback map $U^*$ in \eqref{eq:optimal-feedback-maps}.

The second benchmark problem (ii) with unobservable $\mu$ and without information
acquisition is obtained by restricting the set of controls $(u,h)\in\cA$ to $h\equiv 0$.
The value function of this problem is therefore given by
\[
  V^{no}(z) := \inf_{(u,0)\in\cA} \cJ(u,0;z)
  \quad\text{subject to \eqref{eq:state-aux}}
\]
for any $z=(x,m_0,\gamma_0)\in\R^2\times[0,\infty)$.
From the discussion in Section~\ref{sec:reduction}, it follows that
\[
  V^{no}(x,m,\gamma_0)
	= a_{1}x^{2} + a_{2}m^{2} + a_{3}xm
	+ b_{1}x + b_{2}m + c^{no}(\gamma_0),
\]
where $a_1,a_2,a_3,b_1,b_2$ are still as above and where $c^{no}$ solves
\eqref{eq:HJB-reduced-problem-unreduced} after setting $h\equiv 0$.
Again, the optimal feedback map coincides with $U^*$ given by
\eqref{eq:optimal-feedback-maps}.

It is intuitive that more information on the state comes with lower expected cost,
and in fact we see that
\[
  V^{full}(x,m) \leq V(x,m,\gamma_0) \leq V^{no}(x,m,\gamma_0),
  \quad (x,m,\gamma_0)\in\R^2\times[0,\infty).
\]
In fact, from the semi-explicit expressions we obtain the following relation for the value functions
\[
  V^{full}(x,m) - V(x,m,\gamma_0) = c^{full} - v(\gamma_0)
  \quad\text{and}\quad
  V(x,m,\gamma_0) - V^{no}(x,m,\gamma_0) = v(\gamma_0) - c^{no}(\gamma_0)
\]
for all $(x,m,\gamma_0)\in\R^2\times[0,\infty)$.
It should be highlighted that in the comparison of $V^{full}$ and $V$
there may be an additional difference
as, in practice, $V^{full}$ has to be evaluated at the true initial state $\mu_0$
of $\mu$, whereas $V$ is evaluated at the estimated initial value $m$ of $\mu$.
Conversely, we see that the difference $v-c^{no}$ is exactly the additional value
provided by the ability to acquire information.
This underlines once again our interpretation of $\gamma_0$ as a measure of uncertainty
and allows us to interpret $v$ as a measure of the {\it value of information}.

\begin{figure}[ht]
\centering
\includegraphics[width=0.5\linewidth]{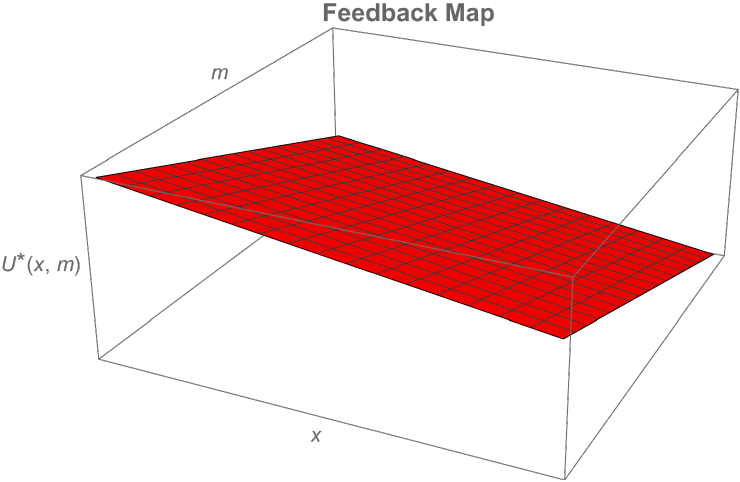}
\caption{Illustration of the optimal feedback map $U^*$.}
\label{fig:original-feedback-map}
\end{figure}

The fact that all three problems feature the same optimal feedback map $U^*$
as illustrated in Figure~\ref{fig:original-feedback-map}
shows that in all cases, the optimal state control satisfies the certainty equivalent principle.
Indeed, the acquisition of information through $h$ and the corresponding level
of uncertainty $\gamma^h$ do not directly affect the optimal state
control $u^*$, but only indirectly through the estimate $m^{u,h}$ of the hidden
state $\mu$ via the feedback map $U^*$.

\begin{figure}[ht]
\centering
\begin{subfigure}{.49\textwidth}
	\centering
	\includegraphics[width=0.8\linewidth]{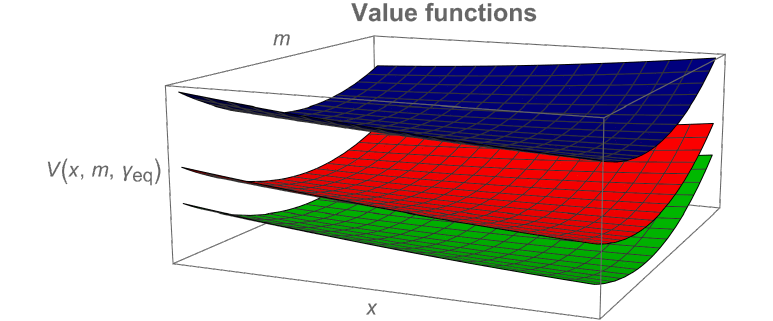}
	\caption{Comparison of value functions}
	\label{fig:original-problem-value-function}
\end{subfigure}
\begin{subfigure}{.49\textwidth}
	\centering
	\includegraphics[width=0.8\linewidth]{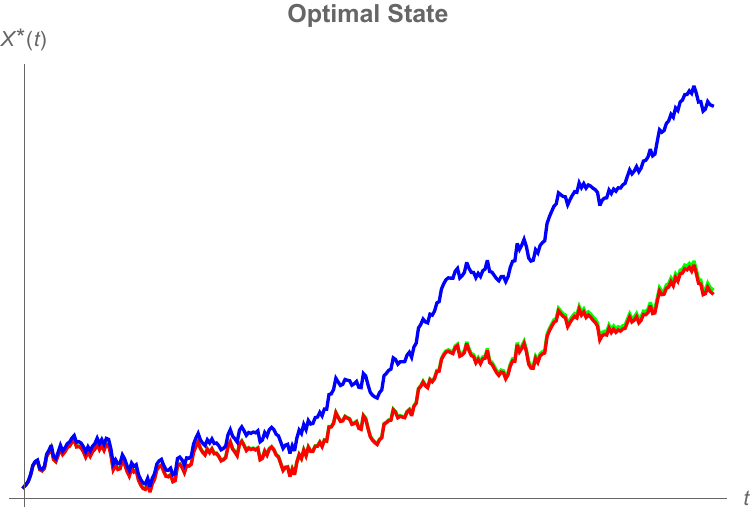}
	\caption{Comparison of optimally controlled states}
	\label{fig:original-state-comparison-1}
\end{subfigure}
\par\bigskip
\begin{subfigure}{.49\textwidth}
	\centering
	\includegraphics[width=0.8\linewidth]{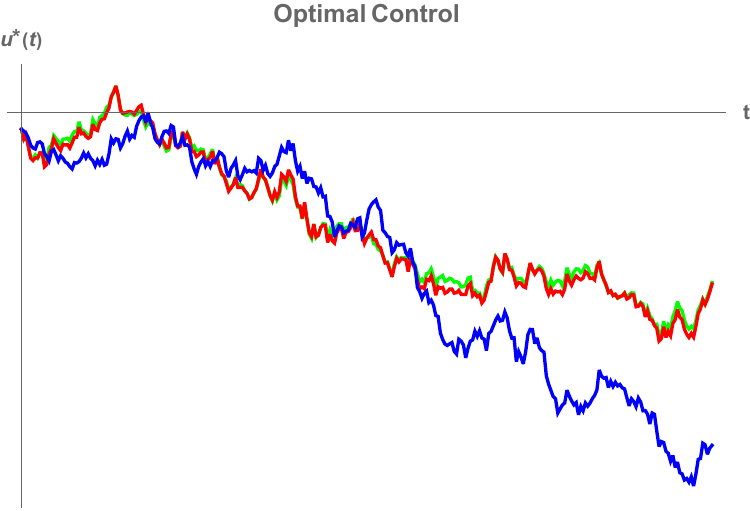}
	\caption{Realizations of the optimal control $u^*$}
	\label{fig:realizations-original-control-1}
\end{subfigure}
\begin{subfigure}{.49\textwidth}
	\centering
	\includegraphics[width=0.8\linewidth]{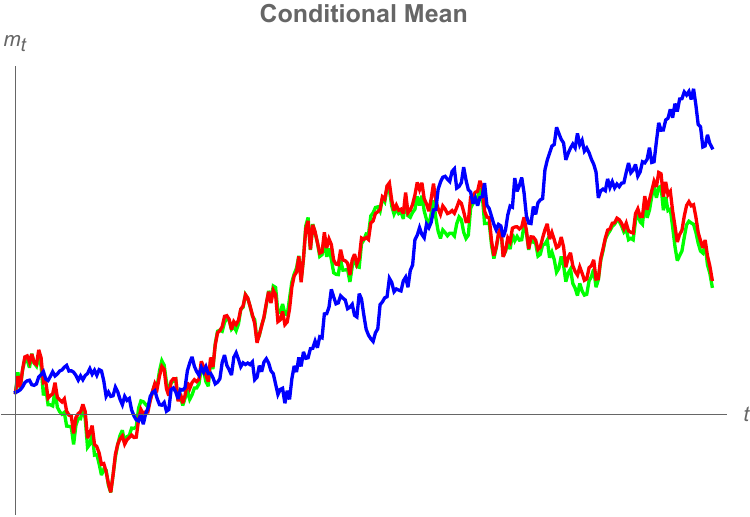}
	\caption{Quality of the estimation of $\mu$}
	\label{fig:original-cm-comparison}
\end{subfigure}
\caption{Illustrative figures for the original control problem.
Green colors represent the fully observable problem,
red colors the partial observation problem with information acquisition,
and blue colors represent the partial observation problem without information acquisition.}
\label{fig:original-problem-figures}
\end{figure}

This observation has interesting consequences regarding the relation between
the value functions and the optimally controlled states in the three problems
under consideration.
In Figure~\ref{fig:original-problem-value-function} we compare the value functions
$V^{full}$ (green), $V$ (red), and $V^{no}$ (blue) at the equilibrium conditional
variance $\equi$ for different value of $(x,m)$.
In Figure~\ref{fig:original-state-comparison-1}, we compare the evolution of
the optimally controlled state $X^*$, in Figure~\ref{fig:realizations-original-control-1},
where we compare the evolution of the optimal control $u^*$, and
Figure~\ref{fig:original-cm-comparison}, where we compare the evolution of $\mu$
and the filter conditional means $m$.
In all three figures,
the green trajectories correspond to the full observation problem,
the red trajectories to the partial observation problem with information acquisition,
and the blue trajectories to the problem with partial observations and no information acquisition.

\subsection{Sensitivity Analysis}\label{subsec:sensitivity}
We conclude our discussion of the qualitative properties of the control problem
with a sensitivity analysis of the equilibrium $\equi$,
the optimal control $\hEqui$ at the equilibrium, and the value $\vEqui$ at the equilibrium
with respect to some of the model parameters
and with respect to a parametric change of the cost function $c$ to $\alpha c$ with $\alpha>0$.
Recall that some properties of $\equi$, $\hEqui$, and $\vEqui$ where already
obtained in Proposition~\ref{prop:properties-equilibrium-state-control}
and Proposition~\ref{prop:value-equilibrium}.
Let us also recall that $\bar{\sigma}_1 := 1/\sigma_1$.

Note that the dependence of the optimal information acquisition rate on the
model parameters comes from plugging the derivative $v'$ of the value function
into the feedback map $H^*$. 
To circumvent analyzing the dependence of the value function and its derivative
on the parameters and as we are interested in the sensitivity analysis at the
equilibrium, it turns out to be convenient to introduce the adjoint state $p$
coming from the Pontryagin maximum principle and study the adjoint system at the
equilibrium level $\equi$.
From \cite[Theorem 3.1, Corollary 3.1]{ye1993nonsmooth}, for any initial
condition $\gamma_0\in[0,\infty)$, the adjoint state $p$ is given by 
\begin{equation}\label{eq:adjoint-sensitivity-relation}
  p_{t}=v'(\gamma^{*;\gamma_0}_{t})\in [0,L_{v}].
\end{equation}
This, together with Proposition~\ref{prop:properties-equilibrium-state-control},
also implies that the adjoint state at the equilibrium is constant in time,
and we denote this value by $\pEqui = v'(\equi)$.
Finally, the adjoint state satisfies
\[
  \rd p_{t} = \Bigl[\bigl(2(\bar{\sigma}_{1}^{2} + h^{*;\gamma_0}_{t})\gamma^{*;\gamma_0}_{t}
    + 2\lambda + \delta\bigr)p_{t} - \bar{a}\Bigr] \rd t,
  \quad t\in[0,\infty),
  \quad p_{0} = v'(\gamma_0),
\]
and, by Theorem \ref{th:verification-deterministic} and \eqref{eq:adjoint-sensitivity-relation},
the optimal information acquisition rate satisfies
$h^{*}_{t} = \fh((\gamma_{t}^{*;\gamma_0})^{2}p_{t})$.
In order to study the quantities of interest with respect to the model parameters,
we proceed as follows.
First, with a slight abuse of notation, we
introduce the function $\Phi:\R^{2}\to\R^{2}$ consisting of the
coefficient functions of the state and the adjoint, that is
\begin{align*}
  \Phi(\gamma,p;\theta) := \begin{pmatrix}
    -\bigl(\bar{\sigma}_{1}^{2}+\fh(\gamma^{2}p)\bigr) \gamma^{2}
	  -2\lambda\gamma +\sigma_{2}^{2}\\
    \bigl(2(\bar{\sigma}_{1}^{2} + \fh(\gamma^{2}p))\gamma
	  + 2\lambda +\delta\bigr)p - \bar{a}
  \end{pmatrix},
  \quad (\gamma,p)\in\R^2,
\end{align*}
where $\theta$ serves as a placeholder for the model parameter of interest.
From the existence of the equilibrium $\equi$, we know that for any parameter
$\theta$ there exists a unique solution
$(\equi,\pEqui)$ satisfying $\Phi(\equi,\pEqui;\theta)=0$.
In particular, the function $\phi:(0,\infty)\to\R^{2}$ which maps a parameter $\theta$
to the associated equilibrium, that is $\phi(\theta):=(\equi,\pEqui)$, is
well-defined.
The function $\phi$ allows us to study the sensitivity of the equilibrium with respect to the model parameters.
Differentiability and expressions for the derivative of $\phi$ with respect to any model parameter
$\theta$ are obtained from the implicit function theorem \cite[Theorem 1B.1]{dontchev2009implicit}.
Having computed the derivatives,
the sign of the derivatives yield the direction of change of $(\equi,\pEqui)$
with respect to a change in the model parameter $\theta$.
The sensitivity of $\hEqui$ with respect to $\theta$ can then be deduced from
$\hEqui = \fh(\equi^2\pEqui)$, and the sensitivity of $\vEqui$ with respect
to $\theta$ can be deduced from \eqref{eq:value-function-at-equilibrium}.

\begin{table}[ht]
\centering
\begin{tabular}{|c|c|c|c|}
	\hline
	$\theta$
	& $\frac{\rd \equi}{\rd \theta}$
	& $\frac{\rd \hEqui}{\rd \theta}$
	& $\frac{\rd \vEqui}{\rd \theta}$\\
	\hline
	$\sigma_{1}$ & $+$ & $+$ & $+$\\
	\hline
	$\sigma_{2}$ & $+$ & $+$ & $+$\\
	\hline
	$\kappa$ & $-$ & $+$ & $?$\\
	\hline
	$\alpha$ & $+$ & $-$ & $+$\\
	\hline
\end{tabular}
\caption{Signs of the derivatives of the equilibrium state $\equi$,
  the optimal control at the equilibrium $\hEqui$,
  and the value function at the equilibrium $\vEqui$
  with respect to the parameters of interest.}
\label{tab:sensitivity}
\end{table}

Since the calculations are rather tedious and do not add much value to the general
understanding of the model, we only report the signs of the respective derivatives in Table~\ref{tab:sensitivity}
for the model parameters $\sigma_1$, $\sigma_2$, $\kappa$,
and for the parametric change of the cost function from $c$ to $\alpha c$ for
$\alpha>0$ and postpone more details to Appendix \ref{app:sensitivity-analysis}.
We remark here that the expression obtained
for the derivative of $\vEqui$ with respect to $\kappa$ turns out to be too
complicated to infer the sign of the derivative.
From the results in Table~\ref{tab:sensitivity}
we see that the dependence of the quantities of interest on the volatility
parameters $\sigma_1$ and $\sigma_2$ are not surprising.
Indeed, injecting more noise into the system worsens the quality of estimating
$\mu$ and hence increases the uncertainty.
As a consequence, the equilibrium uncertainty increases and so does the control effort.
Hence, the value of the control problem at the equilibrium increases.
Moreover, let us highlight that the positive
sign of the derivative of $\equi$ with respect to $\sigma_2$ in combination with
Proposition~\ref{prop:properties-equilibrium-state-control} shows that the
convergence $\equi\to 0$ as $\sigma_2\to 0$ is in fact monotone.

Regarding sensitivity with respect to the cost coefficient $\kappa$ of the state $X$, we obtain that
a higher penalty for deviation of the state $X$ leads to a decrease in the equilibrium uncertainty.
Hence, the \DM\ prefers to offset the higher future cost due to insufficient control
by reducing the uncertainty at the equilibrium.
They do so by increasing the control effort.

Finally, for higher cost of control, the uncertainty at the equilibrium increases.
This is due to the fact that in light of higher cost of control, the \DM\
prefers to reduce the control effort, and hence uncertainty increases.
Hence, the overall value of the control problem should be increasing in the control cost.

The sensitivity analysis for the reduced problem can easily be transferred to the
full information problem.
Indeed, the dependency of $\equi$ and $\hEqui$ on the model parameters is trivially identical.
Regarding the sensitivity of the value
function $V$ of the full information problem at the equilibrium, one can use that according to
\eqref{eq:hjb-solution} the value functions $V$ and $v$ are related through
\[
  V(x,m,\gamma) = a_1 x^2 + a_2m^2 + a_3 xm + b_1 x + b_2m + a_2\gamma
    + \frac{1}{\delta}\bigl(C_1 + a_2\sigma_2^2\bigr) + v(\gamma),\quad (x,m,\gamma)\in\S,
\]
where we recall that $C_1 = \sigma_{1}^{2}a_{1} + \lambda \bar{\mu} b_{2} -b_{1}^{2}/({2\rho})$.
Since $a_1,a_2,a_3,b_1,b_2$ do not depend on the parameters $\sigma_1,\sigma_2$
and the cost function $c$, it follows that
\[
  \frac{\rd}{\rd \theta} V(x,m,\equi) = \frac{\rd}{\rd \theta} v(x,m,\equi)
  + \frac{\rd}{\rd \theta}\Bigl( a_{2}\equi
  + \frac{1}{\delta}\Bigl(\sigma_{1}^{2}a_{1}
	  +\lambda \bar{\mu} b_{2}
	  -\frac{b_{1}^{2}}{2\rho}
	  + a_{2}\sigma_{2}^{2}\Bigr)\Bigr)
\]
for $\theta\in\{\sigma_1,\sigma_2,\alpha\}$.
From this and using that $a_1,a_2>0$, it follows that, just as $v(\equi)$,
also $V(x,m,\equi)$ is increasing in all three parameters $\sigma_1,\sigma_2,\alpha$.

\appendix
\section{Calculations for Sensitivity Analysis}\label{app:sensitivity-analysis}
Here, we give the calculations required for the sensitivity analysis of Chapter \ref{subsec:sensitivity}.
We evaluate all expressions at $\gamma = \equi, p = \pEqui$,
sometimes suppressed for notational simplicity.
In order to apply the implicit function theorem, we are required to check that
the Jacobian with respect to the target variables $(\gamma,p)$ at the equilibrium
is invertible.
Note that the Jacobian is given by
\begin{align*}
	\rD_{(\gamma,p)}\Phi
	=\begin{pmatrix}
		f_{\gamma} + f_{h}\psi_{\gamma}
		& f_{h}\psi_{p} \\
		(-f_{\gamma\gamma} - f_{\gamma h}\psi_{\gamma})p
		& \delta - f_{\gamma} - f_{\gamma h}\psi_{p}p
	\end{pmatrix},
\end{align*}
where we write $\psi(\gamma,p):=\fh(\gamma^{2}p)$ for ease of notation.
A quick calculation shows that
\begin{align*}
 f_{\gamma} &= -2(\bar{\sigma}_{1}^{2} + \psi)\gamma - 2\lambda <0,&
 f_{h} &= -\gamma^{2} \leq 0,&\\
 f_{\gamma\gamma} &= -2(\bar{\sigma}_{1}^{2} + \psi) <0,&
 f_{\gamma h} &= -2\gamma \leq0,&\\
 \psi_{\gamma} &= 2\fh'(\gamma^{2}p) \gamma p \geq0,&
 \psi_{p} &= \fh'(\gamma^{2}p) \gamma^{2} \geq0,&
\end{align*}
from which one conclude that the determinant of the Jacobian at equilibrium is given by
\begin{equation*}\begin{split}
 \det \rD_{(\gamma,p)} \Phi(\equi,\pEqui)
 &= \frac{1}{\sigma_{1}^{4}}\Bigl(
 - 4 \equi^{2} \fh(\equi^{2} \pEqui)^{2} \sigma_{1}^{4}\\
 &\hspace{1.5cm}- 6 \sigma_{1}^{2} \Bigl[\Bigl(\equi^{3} \fh'(\equi^{2}\pEqui) \pEqui
   + \frac{\delta}{3}+\frac{4 \lambda}{3}\Bigr) \sigma_{1}^{2}+\frac{4\equi}{3}\Bigr]
   \equi \fh(\equi^{2}\pEqui)\\
 &\hspace{1.5cm}+\Bigl(-2 \pEqui \fh' (\equi^{2}\pEqui) (\delta + 4\lambda) \equi^{3}-2 \lambda (\delta +2 \lambda )\Bigr) \sigma_{1}^{4}\\
 &\hspace{1.5cm}-6 \Bigl(\equi^{3} \fh' (\equi^{2}\pEqui) \pEqui +\frac{\delta}{3}+\frac{4 \lambda}{3}\Bigr) \equi  \sigma_{1}^{2}-4 \equi^{2}\Bigr)<0.
\end{split}\end{equation*}
In particular, the Jacobian is invertible and the implicit function theorem applies.
Next, we obtain the following expression for the derivative
\begin{align*}
   \phi'(\theta) = -(\rD_{(\gamma,p)}\Phi(\equi,\pEqui,\theta))^{-1}
   \rD_{\theta}\Phi(\equi,\pEqui,\theta),
\end{align*}
where the inverse is given by
\begin{align*}
	(\rD_{(\gamma,p)}\Phi(\equi,\pEqui,\theta))^{-1}
	&= \frac{1}{\det \rD_{(\gamma,p)}\Phi(\equi,\pEqui,\theta)}
	\begin{pmatrix}
		\delta - f_{\gamma} - f_{\gamma h}\psi_{p}\pEqui
		& -f_{h}\psi_{p} \\
		(f_{\gamma\gamma} + f_{\gamma h}\psi_{\gamma})\pEqui
		& f_{\gamma} + f_{h}\psi_{\gamma}
	\end{pmatrix}
	=\begin{bmatrix}
		\geq 0 & \geq 0\\
		\leq 0 & \leq 0
	\end{bmatrix}.
\end{align*}
For $\theta=\sigma_{2}^{2}$, we obtain
\begin{align*}
	\rD_{\sigma_{2}}\Phi(\equi,\pEqui,\sigma_{2})
	= (1,0)
\end{align*}
and hence
\begin{equation}\begin{split}\label{eq:equilibrium-sensitivity}
	\frac{\rd (\equi,\pEqui)}{\rd \sigma_{2}^{2}}
	= \phi'(\sigma_{2}) = -\frac{1}{\det \rD_{(\gamma,p)}\Phi(\equi,\pEqui,\theta)}
	\begin{pmatrix}
		\delta - f_{\gamma} - f_{\gamma h}\psi_{p}\pEqui \\
		(f_{\gamma\gamma} + f_{\gamma h}\psi_{\gamma})\pEqui
	\end{pmatrix}
\end{split}\end{equation}
having a positive first entry and a negative second entry,
and hence $\equi$ is strictly increasing in $\sigma_{2}^{2}$ and thus also in $\sigma_{2}$.
For the control equilibrium value $\hEqui$, we obtain
\begin{equation}\begin{split}\label{eq:control-sensitivity}
	\frac{\rd}{\rd \sigma_{2}^{2}}\hEqui &= \fh'(\equi^{2}\pEqui)
	\Bigl( 2 \equi \pEqui\frac{\rd \equi}{\rd \sigma_{2}^{2}}
	+ \equi^{2}\frac{\rd \pEqui}{\rd \sigma_{2}^{2}}\Bigr)\\
	&= -\frac{\fh'(\equi^{2}\pEqui)}{\det \rD_{(\gamma,p)}\Phi(\equi,\pEqui,\theta)}
	   (2 \equi \pEqui(
		\delta + 2(
		\bar{\sigma}_{1}^{2} + \fh(\equi^{2}\pEqui))\equi + 2\lambda
		+ 2\equi\fh'( \equi^{2}\pEqui)\equi^{2}\pEqui)\\
		&\quad+ \equi^{2}\pEqui(
		-2(\bar{\sigma}_{1}^{2} + \fh(\equi^{2}\pEqui))
		-2\gamma\fh'(\equi^{2}\pEqui)2\equi\pEqui))\\
	&= -\frac{\fh'(\equi^{2}\pEqui)}{\det \rD_{(\gamma,p)}\Phi(\equi,\pEqui,\theta)}
		2 \equi \pEqui(
		\delta + (\bar{\sigma}_{1}^{2} 
           + \fh(\equi^{2}\pEqui))\equi + 2\lambda )>0,
\end{split}\end{equation}
as the determinant is strictly negative.
Analyzing the value, from \eqref{eq:value-function-at-equilibrium}, we obtain
\begin{equation*}
	\frac{\rd v(\equi)}{\rd \sigma_{2}^{2}}
	= \frac{1}{\delta}\Bigl(\frac{a_{3}^{2}}{2\rho}\frac{\rd \equi}{\rd \sigma_{2}^{2}}
	+ c'(\bar{h})\frac{\rd \hEqui}{\rd \sigma_{2}^{2}}\Bigr)>0,
\end{equation*}
which as a consequence of \eqref{eq:equilibrium-sensitivity} and \eqref{eq:control-sensitivity}
is strictly increasing in $\sigma_{2}^{2}$.

The expressions for the other parameters are given by
{\tiny
\begin{equation}\begin{split}
	\label{eq:implicit-derivatives-state}
	-\det \rD_{(\gamma,p)} \Phi
	\frac{\rd (\equi,\pEqui)}{\rd \sigma_{1}}
	&=\left[\begin{array}{c}
		\frac{4 \equi^{2} \big(\fh(\equi^{2} \pEqui) \equi \sigma_{1}^{2}+\big(\frac{\delta}{2}+\lambda \big) \sigma_{1}^{2}+\equi \big)}{\sigma_{1}^{5}} 
		\\
		\frac{4 \pEqui \equi \big(\fh(\equi^{2} \pEqui) \equi \sigma_{1}^{2}+2 \lambda  \sigma_{1}^{2}+\equi \big)}{\sigma_{1}^{5}} 
	\end{array}\right]
	=\begin{bmatrix}
		\geq 0 \\\geq 0 
	\end{bmatrix}\\
	-\det \rD_{(\gamma,p)} \Phi
	\frac{\rd (\equi,\pEqui)}{\rd k}
	&=\left[\begin{array}{c}
		-\frac{4 \fh'(\equi^{2} \pEqui) \equi^{4} \big(\delta  \rho_{1} -\sqrt{\delta^{2} \rho_{1}^{2}+4 k \rho_{1}}\big) \rho_{1}^{3} (\delta +\lambda)}{\sqrt{\delta^{2} \rho_{1}^{2}+4 k \rho_{1}} \big(\sqrt{\delta^{2} \rho_{1}^{2}+4 k \rho_{1}}+\rho_{1} (\delta +2 \lambda)\big)^{3}} 
		\\
		-\frac{8 \big(\pEqui \equi^{3} \sigma_{1}^{2} \fh'(\equi^{2} \pEqui) +\fh(\equi^{2} \pEqui) \equi \sigma_{1}^{2}+\lambda  \sigma_{1}^{2}+\equi \big) \big(-\delta  \rho_{1} +\sqrt{\delta^{2} \rho_{1}^{2}+4 k \rho_{1}}\big) \rho_{1}^{3} (\delta +\lambda)}{\sqrt{\delta^{2} \rho_{1}^{2}+4 k \rho_{1}} \sigma_{1}^{2} \big(\sqrt{\delta^{2} \rho_{1}^{2}+4 k \rho_{1}}+\rho_{1} (\delta +2 \lambda)\big)^{3}} 
	\end{array}\right]
	=\begin{bmatrix}
		\leq 0 \\\leq 0 
	\end{bmatrix}\\
	-\det \rD_{(\gamma,p)} \Phi
	\frac{\rd (\equi,\pEqui)}{\rd \alpha}
	&=\left[\begin{array}{c}
	\frac{2 \Bigl(h \Bigl(\frac{\equi^{2} \pEqui}{\alpha}\Bigr) \gamma \sigma_{1}^{2}+\Bigl(\frac{\delta}{2}+\lambda \Bigr) \sigma_{1}^{2}+\equi \Bigr) \fh'\Bigl(\frac{\equi^{2} \pEqui}{\alpha}\Bigr) \equi^{4} \pEqui}{\sigma_{1}^{2} \alpha^{2}} 
	\\
	\frac{2 \pEqui^{2} \fh'\Bigl(\frac{\equi^{2} \pEqui}{\alpha}\Bigr) \equi^{3} \Bigl(\fh \Bigl(\frac{\equi^{2} \pEqui}{\alpha}\Bigr) \equi \sigma_{1}^{2}+2 \lambda  \sigma_{1}^{2}+\equi \Bigr)}{\alpha^{2} \sigma_{1}^{2}} 
	\end{array}\right]
	=\begin{bmatrix}
		\geq 0\\ \geq 0
	\end{bmatrix},
\end{split}\end{equation}}%
where we multiplied by $-\det \rD_{(\gamma,p)} \Phi > 0$ to simplify the expressions.
Using this together with \eqref{eq:control-sensitivity} we obtain
{\tiny\begin{equation}\begin{split}
	-\det\rD_{(\gamma,p)} \Phi
	\frac{\rd \hEqui}{\rd \sigma_{1}}
	&= \frac{12 \pEqui \equi^{3} \big(\fh(\equi^{2} \pEqui) \equi \sigma_{1}^{2}+\frac{\big(\delta +4 \lambda \big) \sigma_{1}^{2}}{3}+\equi \big)}{\sigma_{1}^{5}}<0,\\
	-\det \rD_{(\gamma,p)} \Phi
	\frac{\rd \hEqui}{\rd k}
	&= -\frac{8 \equi^{2} \big(-\delta  \rho_{1} +\sqrt{\delta^{2} \rho_{1}^{2}+4 k \rho_{1}}\big) \rho_{1}^{3} \big(\delta +\lambda \big) \big(\fh(\equi^{2} \pEqui) \equi \sigma_{1}^{2}+\lambda  \sigma_{1}^{2}+\equi \big)}{\sqrt{\delta^{2} \rho_{1}^{2}+4 k \rho_{1}} \big(\sqrt{\delta^{2} \rho_{1}^{2}+4 k \rho_{1}}+\rho_{1} \big(\delta +2 \lambda \big)\big)^{3} \sigma_{1}^{2}}>0,\\
	-\det \rD_{(\gamma,p)} \Phi
	\frac{\rd \hEqui}{\rd \alpha}
	&= \biggl(\frac{36 \sigma_{1}^{2} \equi^{3} 
	\Bigl(\fh  \Bigl(\frac{\equi^{2} \pEqui}{\alpha}\Bigr) \equi \sigma_{1}^{2}+\frac{(\delta +4 \lambda ) \sigma_{1}^{2}}{3}+\equi \Bigr) \pEqui \fh' \Bigl(\frac{\equi^{2} \pEqui}{\alpha}\Bigr)}{\alpha^{4} \sigma_{1}^{6}}\\
    & \qquad\qquad
	+ \frac{2 \Bigl(\fh \Bigl(\frac{\equi^{2} \pEqui}{\alpha}\Bigr) \equi \sigma_{1}^{2}+\Bigl(\frac{\delta}{2}+\lambda \Bigr) \sigma_{1}^{2}+\equi \Bigr) \Bigl(\fh \Bigl(\frac{\equi^{2} \pEqui}{\alpha}\Bigr) \equi \sigma_{1}^{2}+\lambda  \sigma_{1}^{2}+\equi \Bigr) \alpha}{3}\biggr)\\
	& \quad \times \pEqui
	\equi^{2} \Bigl(\equi^{3} \Bigl(\fh  \Bigl(\frac{\equi^{2} \pEqui}{\alpha}\Bigr) \equi \sigma_{1}^{2}
	+\frac{(\delta +4 \lambda ) \sigma_{1}^{2}}{3}+\equi \Bigr) \pEqui \fh' \Bigl(\frac{\equi^{2} \pEqui}{\alpha}\Bigr)+\frac{\alpha \sigma_{1}^{2}}{6}\Bigr) <0.
\end{split}\end{equation}}
The sign of the derivative of the value function at the equilibrium with respect $\sigma_{1}$
is the same as for $\sigma_{2}$.
For the derivative of $\bar{v}$ with respect to $\alpha$, it is clear that the optimal
value at the equilibrium $\bar{v}$ is increasing in $\alpha$.

\bibliographystyle{alpha}
\bibliography{preprint}

\end{document}